\documentclass[12pt]{amsart}
\usepackage{amsmath,amssymb,latexsym, amsfonts, amscd, amsthm}
\usepackage{graphicx,epsfig,color,hyperref}

\theoremstyle{plain}
\newtheorem{theorem}{Theorem}[section]

\newtheorem{proposition}[theorem]{Proposition}
\newtheorem{corollary}[theorem]{Corollary}

\newtheorem{assumption}[theorem]{Assumption}
\newtheorem{lemma}[theorem]{Lemma}

\theoremstyle{definition}
\newtheorem{definition}[theorem]{Definition}
\newtheorem{remark}[theorem]{Remark}

\newcommand{\CC}{{\mathbb C}}
\newcommand{\FF}{{\mathbb F}}
\newcommand{\GG}{{\mathbb G}}
\newcommand{\QQ}{{\mathbb Q}}
\newcommand{\RR}{{\mathbb R}}
\newcommand{\ZZ}{{\mathbb Z}}
\newcommand{\WW}{{\mathbb W}}

\newcommand{\cH}{{\mathcal H}}

\newcommand{\cF}{{\mathcal F}}

\newcommand{\cS}{{\mathcal S}}

\newcommand{\cO}{{\mathcal O}}

\newcommand{\cX}{{\mathcal X}}

\newcommand\Lie{{\rm Lie}}
\newcommand\Sch{{\rm Sch}}

\begin{document}

\title{On two mod $p$ period maps: Ekedahl--Oort   and fine Deligne--Lusztig stratifications}
\author{Fabrizio Andreatta}
\address{Dipartimento di Matematica ``F. Enriques'', Universit{\`a} degli Studi di Milano\\ via C. Saldini, 50\\ Milano I-20133\\ Italy}\email{Fabrizio.Andreatta@unimi.it}
\maketitle

\begin{abstract}  Consider a Shimura variety of Hodge type admitting a smooth integral model $S$ at an odd prime $p\geq 5$.  Consider its perfectoid cover  $S^{\rm ad}(p^\infty)$ and the Hodge-Tate period map introduced by A.~Caraiani and P.~Scholze. We compare the pull-back to $\cS(p^\infty)^{\rm ad}$ of the Ekedahl-Oort stratification on the mod $p$ special fiber of $S$   and the pull back to $\cS(p^\infty)^{\rm ad}$ of the  fine Deligne-Lusztig stratification on the mod $p$ special fiber of the flag variety which is the target of the Hodge-Tate period map. An application to the non-emptiness of Ekedhal-Oort strata is provided.

\end{abstract}

\tableofcontents

\section{Introduction}

Let $(G,X)$ be a Shimura datum; in particular, $G$ is a connected reductive group over $\QQ$ and $X$ is a $G(\RR)$-conjugacy class of a homomorphism of algebraic groups  $h\colon {\rm Res}_{\CC/\RR}(\GG_m)\to G_\RR$.
Denote by $\mathrm{Sh}_K(G,X)$ the associated Shimura variety over $\CC$, for some compact open subgroup $K\subset G(\mathbb{A}^f)$ of the finite adelic points of $G$. We fix an odd prime $p$ and we assume that:

\begin{itemize}

\item[i.] $K$ is hyperspecial at $p$ i.e., $K=K_p K^p$ with $K^p\subset G(\mathbb{A}^{f,(p)})$ 
a compact open subgroup of the prime-to-$p$, finite  adelic points of $G$ and $K_p=G_{\ZZ_p}(\ZZ_p)$, the group of $\ZZ_p$-points of  some reductive group $G_{\ZZ_p}$ over $\ZZ_p$ with generic fiber $G_{\QQ_p}$;

\item[ii.] $K$ is neat;

\item[iii.] $(G,X)$ is of Hodge type, that is, there is an embedding of Shimura data $(G,X)\subset ({\rm GSp}_{2g}, \mathbb{H}_g)$ with ${\rm GSp}_{2g}$ the group of symplectic similitudes on a $\QQ$-vector space of dimension $2g$, endowed with a non-degenerate symplectic pairing and with $\mathbb{H}_g$ the double Siegel  space.

\end{itemize}

Results of Shimura and Deligne guarantee that $\mathrm{Sh}_K(G,X)$ admits a canoncial model over the reflex field $E$ which is a number field, defined as the minimal subfield of $\CC$ of definition of the conjugacy class $Z$ of the Hodge cocharacter  $$\mu_h\colon \GG_{m,\CC}\stackrel{{\rm Id}\times 1}{\longrightarrow} \GG_{m,\CC} \times \GG_{m,\CC} \cong {\rm Res}_{\CC/\RR}(\GG_m)_\CC\stackrel{h_\CC}{\longrightarrow} G_\CC.$$

Fix a prime $v$ of $E$ over the prime $p$; let $\kappa$ be the residue field of $\cO_E$ at $v$. Thanks to our assumptions $v$ is unramified over $p$ so that $\widehat{\cO}_{E,v}\cong \WW(\kappa)$ and $Z$ admits a $\WW(\kappa)$-valued point that we denote $\mu$.

Denote by $S$ be the integral canonical model of $\mathrm{Sh}_K(G,X)$ over $\cO_{E,v}$ constructed by Kisin \cite{Kisin}. Given a finite and admissible rational cone decomposition $\Sigma$ for $G$, we denote by $S^\Sigma$  the associated toroidal compactification constructed in \cite{MP}. We fix one and we suppose that $S^\Sigma$ is smooth and projective over $\cO_{E,v}$ (and certain other technical assumptions for $\Sigma$ hold; see (\ref{ass})).  We let $S_0^\Sigma$ be the mod $p$ special fiber. Write $\cS^\Sigma$ for the $p$-adic formal scheme defined by completing $S^\Sigma$ along  $S_0^\Sigma$ and  $\cS^{\Sigma, \rm ad} $ for its adic generic fiber (in the sense of Huber). We denote by $S^{\rm ad}\subset \cS^{\Sigma, \rm ad}$ the open, adic subspace defined by the scheme $S_{\QQ_p}$.

We also let $\mathrm{F\ell}_{G,\mu^{-1}}$ be the flag variety over $\WW(\kappa)$ classifying parabolic subgroups of $G_{\ZZ_p}$ of type $\mu^{-1}$. Set $\mathrm{F\ell}_{G,\mu^{-1},0}$ to be the special fiber over $\kappa$, $\cF\ell_{G,\mu^{-1}}$ to be the formal scheme defined by the completion of $\mathrm{F\ell}_{G,\mu^{-1}}$ along $\mathrm{F\ell}_{G,\mu^{-1},0}$ and $\cF\ell_{G,\mu^{-1}}^{\rm ad}$ to be the generic adic fiber of $\cF\ell_{G,\mu^{-1}}$.

Consider the perfectoid tower $\pi\colon S^{\rm ad}(p^\infty)\to S^{\rm ad}$ and the Hodge-Tate period map $$\pi_{\rm HT}\colon S^{\rm ad}(p^\infty) \longrightarrow \mathcal{F}\ell_{G,\mu^ {-1}},$$defined by \cite[Thm. 2.1.3]{CS}. We have the following diagram of topological spaces:

$$\begin{matrix} & & S^{\rm ad}(p^\infty) & & \cr
 & \swarrow & & \searrow & \cr
\cS^{\Sigma, \rm ad} & & & & \cF\ell_{G,\mu^ {-1}}^{\rm ad}\cr
\downarrow & & & & \downarrow\cr
S_0^\Sigma & & & & \mathrm{F}\ell_{G,\mu^{-1},0},\cr
\end{matrix}$$where the bottom vertical arrows are the specialization maps ${\rm sp}\colon \cS^{\Sigma, \rm ad}\to S_0^\Sigma$ and  ${\rm sp}\colon \cF\ell_{G,\mu^{-1}}^{\rm ad}\to \mathrm{F}\ell_{G,\mu^{-1},0}$. We write $q_1\colon  S^{\rm ad}(p^\infty) \to S_0^\Sigma$ and $q_2\colon S^{\rm ad}(p^\infty) \to \mathrm{F}\ell_{G,\mu^{-1},0}$ for the two composite maps. These are the two mod $p$ period maps of the title.

The varieties $S_0^\Sigma$ and $\mathrm{F\ell}_{G,\mu^{-1},0} $ are endowed with natural stratifications. On $S_0^\Sigma$ we have the so called \emph{Ekedahl--Oort stratification} introduced by T.~Ekedhal and F.~Oort in the Siegel case, by  C.~Zhang \cite{Chao} in the general setting of the special fiber of Shimura varieties of Hodge type, such as  $S_0$, and then extended to toroidal compactiofications, such as $S_0^\Sigma$, by W.~Goldring and J.-S.~Koskivirta  \cite{GK}. On the other hand $\mathrm{F\ell}_{G,\mu^{-1},0} $ is endowed with the so called \emph{ fine Deligne-Lusztig stratifications}, defined by X.~He \cite{He}. Thanks to results of  R.~Pink., T.~Wedhorn and P.~Ziegler in \cite{AZD} one knows that:\smallskip

$\bullet$ the strata of these two stratifications   are classified by subsets ${^{I_+}}W$ (in the EO case) and ${^{I_-}}W$ (in the DL case)  of the Weyl group $W$ of $G$. Here, $I_+$ and $I_-$ are the types of the parabolic subgroups $P_{\mu^\pm}$ of $G$ defined by  $\mu$ and $\mu^{-1}$ respectively and   ${^{I_{\pm}}}W$ is the set of Konstant representatives of the quotient sets $W/W_{I^\pm}$ of $W$ by the Weyl group of the standard Levi subgroups of  $P_{I_{\pm}}$; \smallskip

$\bullet$ there exists eplicit order  relations  $\preccurlyeq_{\rm EO}$ on ${^{I_+}}W$ and $\preccurlyeq_{\rm DL}$ on ${^{I_-}}W$  corresponding to the closure relations on the associated Ekedahl-Oort  stratification on  $S_0^\Sigma$ and Deligne-Lusztig stratification on $\mathrm{F\ell}_{G,\mu^{-1},0} $ respectively.\smallskip

$\bullet$  the stratum associated to  $w\in {^{I_{\pm}}W}$ is equidimensional of dimension equal to the length $\ell(w)$ of $w$ in the Weyl group $W$.  

\smallskip

For $w\in {^{I_{\pm}}W}$ we let $C_{\rm EO}(w)$, resp. $C_{\rm DL}(w)$ be the corresponding stratum on  $S_0^\Sigma$, resp.~on $\mathrm{F\ell}_{G,\mu^{-1},0} $.
The natural question we address in this paper is the relation between these two stratifications when pulled-back to $S^{\rm ad}(p^\infty)$. Our main result is the following.  In  Lemma \ref{lemma:dualityonWJ} we prove that the map $$\sigma_0\colon {^{I_+}}W \to {^{I_-}}W, \qquad w\mapsto w_0^{I_+} w,$$given by left multiplication by a suitable element $w_0^{I_+}\in W$, is an order reversing  bijection and $\ell\bigl(\sigma_0(w)\bigr)=\ell\bigl(w_0^{I_+})-\ell(w)$ for every $w\in  {^{I_+}}W$. Here the order relations are  $\preccurlyeq_{\rm EO}$ and $\preccurlyeq_{\rm DL}$ respectively. Then:

\begin{theorem}\label{thm:main} Assume that $p\geq 5$ and that $q_1^{-1}\bigl({C_{\rm EO}(w')}\bigr) \cap  q_2^{-1}\bigl(C_{\rm DL}(w)\bigr) $ is non-empty  for some $w'\in {^{I_+}}W$ and $w\in {^{I_-}}W$. Then $w\preccurlyeq_{\rm DL} \sigma_0(w')$ (equivalently $w'\preccurlyeq_{\rm EO} \sigma_0(w)$).

\end{theorem}

This should be seen as the discrete analogue of the phenomenon pointed out in \cite{CS}, where a stratification (in a loose sense) of $\cF\ell_{G,\mu^ {-1}}^{\rm ad}$ is defined so that its inverse image via $\pi_{\rm HT}$  coincides with the inverse image of the Newton polygon stratification on $S^{\rm ad}$ but the closure relation is  opposite to the one of the Newton stratification. In loc. cit., the classification of vector bundles on the Fargues-Fontaine curve $\cX_{\CC_p^\flat}$ plays a crucial role. Recall that in Fontaines' theory we have a classical period ring $A_{\rm inf}$ and a map $A_{\rm inf}\to \cO_{\CC_p}$ with kernel a principal ideal, of which we fix a generator $\omega$. The Fargues-Fontaine curve is the quotient of ${\rm Spa}\bigl(A_{\rm inf}[p^{-1}\omega^{-1}],A_{\rm inf}\bigr)$ with respect to Frobenius. In our case we need to understand what happens at the non-analytic point of ${\rm Spa}\bigl(A_{\rm inf},A_{\rm inf}\bigr)$ defined by the vanishing of the ideal $(p,\omega)$. The assumption that $p\geq 5$ comes from the technical requirement that a certain open of $\cX_{\CC_p^\flat}$ is non-empty.

\

The EO stratification is constructed by Goldring and Koskivirta using a map of stacks $\zeta^\Sigma\colon S_0^\Sigma\to H_{\overline{\FF}_p}-\hbox{\rm Zip}^\chi$ to the stack of so called  $H$-Zips (see \S\ref{sec:Zips} for this notion) where $H$ is  the mod $p$ reduction of a model of $G$ over $\ZZ_p$ and $\chi$ is the cocharacter of $H$ defined by the mod $p$ reduction of the Hodge cocharacter $\mu$.  A second result, conjectured by Wedhorn and Ziegler in \cite{WZ},  is Theorem \ref{theorem:zetasigmasmooth} stating that 

\begin{theorem}\label{thm:mainthm2} The map $\zeta^\Sigma$ is a smooth morphism of stacks.
\end{theorem}

\

Theorems \ref{thm:main} and \ref{thm:mainthm2}   allow to transfer questions on one stratification to questions on the other stratification. For example,  we prove the following

\begin{corollary}\label{cor:nonempty} Assume that $p\geq 5$. Then, for every $w'\in {^{I_+}}W$ and every $w\in  {^{I_-}}W$, the sets  $q_1^{-1}\bigl({C_{\rm EO}(w')}\bigr)$ and $q_2^{-1}\bigl(C_{\rm DL}(w)\bigr)$ are  non-empty. In fact, $q_1^{-1}\bigl({C_{\rm EO}(w')}\cap S_0\bigr)$ is also non-empty.

\end{corollary}

The result in the EO case has been announced by  C.-F.~Yu \cite{Yu} based on work of D.U.~Lee \cite{Lee} and M.~Kisin \cite{Kisin2}. We prove this result differently by showing  first in Proposition \ref{prop:HTonto} the stronger statement that the map $q_2$ is surjective, and hence that every DL stratum is non-empty. This is elementary  projective geometry. It implies, in particular, that $q_2^{-1}\bigl(C_{\rm DL}({^{I_-}} w_0)\bigr)$ is non-empty, where the element $w_0^{I_+}={^{I_-}} w_0\in {^{I_-}}W$ is the element of maximal length. Due to Theorem \ref{thm:main}  this implies that the zero dimensional stratum $C_{\rm EO}(\mathbf{1})$, corresponding to identity element $\mathbf{1}$,  is non-empty.  We deduce from Theorem \ref{thm:mainthm2}  that the EO stratification is closed under generalization for the order $\preccurlyeq_{\rm EO}$ and, hence, all the EO strata are non-empty as $C_{\rm EO}(\mathbf{1})$ is non-empty. As the map $q_1$ is the composite of the surjective maps $\pi\colon S^{\rm ad}(p^\infty) \to S^{\rm ad} $ and ${\rm sp}\colon S^{\rm ad}\to S_0^\Sigma$, it is also surjective and the first part of the Corollary follows.

In fact, we show in Corollary \ref{cor:pointnotintheimageofzetaSigma} that the stratum $C_{\rm EO}(\mathbf{1})$ is fully contained in $S_0$ verifying   Condition 6.4.2 in \cite{GK}. This implies the second claim of the Corollary.

\section{Algebraic Zip Data}\label{sec:AZD}

We recall the notion of Algebraic Zip Data intoroduced in \cite{AZD}. We follow the notation and conventions of \cite[\S 1.2]{GK}. Let $H$ be a connected reductive group defined over the prime field $\FF_p$. Fix an algebraic closure $k$ of $\FF_p$ and a cocharacter $\chi\colon \GG_{m,k}\to H_k$ over $k$. The pair $(H,\chi)$ is called a cocharacter pair in loc. cit. We denote by $P=P_\chi$ the parabolic subgroup of $H_k$ defined by $\chi$: it is characterized by the fact that $\Lie P\subset \Lie H_k$ is the sum of the non-negative weight spaces for the adjoint action of $\GG_{m,k}$ via $\chi$.

As $H$ is defined over a finite field, there exists a Borel subgroup $B\subset H$, with maximal torus $T$, defined over $\FF_p$. We assume that $B_k\subset P$; this is Assumption (1.2.3) of \cite{GK}. Possibly after conjugating $\chi$ by an element of $H(k)$, this holds. We let $W=W(H_k,T_k)$ be the Weyl group of $H_k$  with respect to $T_k$ and we let $\Delta\subset W$ be the subset of simple reflections  defined by $B_k$. Let $I\subset \Delta$ be the type of $P$. It is the subset of simple reflections generating the Weyl group $W_L$ of the standard  Levi subgroup $L\subset P$, the Levi subgroup containing $T_k$. We then have two cosets representatives $^I W$ and $W^I$ in $W$ of the quotient set $W_I\backslash W$, resp. $W/W_I$: for every coset  $W_Iw\in W_I\backslash W$ the corresponding element  $w'\in{^IW}$ is the unique element $w'\in W_I w$ of shortest length (and similarly for $W^I$ and $W/W_I$).

With the notation and assumptions above, we recall the follwing definition \cite[Def. 3.1]{AZD}.

\begin{definition}\label{AZD} An algebraic zip datum for the cocharacter pair $(H,\chi)$ is a triple $\mathcal{Z}:=(P,Q,\varphi)$ where $P$ is the parabolic subgroup of $H_k$ defined by $\chi$,  $Q$ is another parabolic subgroup of $H_k$ and $\varphi\colon P/R_uP\to Q/ R_u Q$ is an isogeny of the reductive quotients. The group $$E_\mathcal{Z}:=\bigl\{(p,q)\in P\times Q \vert \varphi(\pi_P(p))=\pi_Q(q) \bigr\}$$is called the zip group associated to $\mathcal{Z}$. It acts on $H_k$ on the left as follows $$E_\mathcal{Z} \times H_k \to H_k,\quad (p,q) \cdot g= p g q^{-1} $$

\end{definition}

Thanks to \cite[Prop. 3.7]{AZD} there exists $g\in H(k)$ such that 

\begin{itemize}

\item[a.] $^g B_k\subset Q$,  
\item[b.] $\varphi(\pi_P(B_k))=\pi_Q({^g B_k})$,
\item[c.] $\varphi(\pi_P(T_k))=\pi_Q({^g T_k})$. 

\end{itemize} 

Let $L\subset P$ and $M\subset Q$ be the stantard Levi subgroups defined by $T_k \subset P$ and ${^g T_k}\subset Q$ respectively. 
The isogeny $\varphi$ defines  an isogeny $\varphi\colon L\to M$ such that $\varphi (B_k\cap L)={^g B_k}\cap M$ and $\varphi (T_k)={^g T_k}$. If $J$ is the type of $Q$ this induces an isomorphism 

\begin{equation}\label{psi} \psi\colon (W_I,I)\stackrel{\cong}{\longrightarrow } (W_J,J)\end{equation}
 
of the Weyl groups $W_I\to W_J$ suh that $\Psi(I)=J$.

We also fix representatives $\widetilde{w}\in {\rm N}_{H_k}(T_k)$ for $w\in W ={\rm N}_{H_k}(T_k)/{\rm C}_{H_k}(T_k)$ such that $\widetilde{w_1 w_2}=\widetilde{w}_1 \widetilde{w}_2$ if $\ell(w_1w_2)=\ell(w_1) \ell(w_2)$. This can be done as explained in \cite[\S 2.3]{AZD}.

We assume that $\mathcal{Z}$ is {\it orbitally finite}, that is the number of orbits for the action of $E_{\mathcal{Z}}$ on $H_k$ is finite. This holds, for example, if the differential of the isogeny $\varphi$ at $1$ vanishes. For every $w\in W$ we write $H_w:=E_{\mathcal{Z}} \cdot (\widetilde{w} g^{-1})\subset H_k$ for the $E_{\mathcal{Z}}$-orbit defined by $\widetilde{w} g^{-1}$. Then, \cite[Thm. 7.5]{AZD} states that:

\begin{theorem}\label{thmIW} The map $${^I W} \longrightarrow\left\{ E_{\mathcal{Z}}-\hbox{{\rm orbits in }} H_k\right\},\qquad  w \mapsto H_w$$is a bijection. Morever,

\begin{itemize}

\item[1.] for any $w\in {^I W}$ the corresponding orbit $H_w$ is a locally closed, smooth subvariety of $H_k$ of dimension $\dim P+\ell(w)$.

\item[2.] for any $w$, $w'\in {^I W}$ we have that $H_{w'} \subset \overline{H}_w$, and we write $w' \preccurlyeq w$, if and only if there exists $y\in W_I$ such that $y w' \psi(y)^{-1} \leq w $ for the Bruhat order $\leq$ on $W$.

\end{itemize}

\end{theorem}

\begin{remark} The Assumption that $B\subset P$ is related to the concept of frame in \cite[Def. 3.6]{AZD} for  $\mathcal{Z}$.  If we set $(B',T',g')$ with $B':={^{g^{-1}}}B_k$, $T':={^{g^{-1}}}T_k$ and $g'=g$, the triple $(B',T',g')$ is a frame for $\mathcal{Z}$.    We prefer to use the conventions of \cite{GK} fixing a Borel $B_k\subset P$ instead of a Borel in $Q$. This has the effect, for example, that the statement of \cite[Thm. 7.5]{AZD}  is slightly changed.

\end{remark}

There is a dual presentation of the orbit space  for the action of $E_{\mathcal{Z}}$ on $H_k$, described in \cite[Thm. 11.2]{AZD}, in terms of the subset $W^J\subset W$. As explained in \cite[Prop. 9.13 \& 9.14]{AZD} there exists a length preserving bijection $$\sigma\colon {^I W}\to W^J$$characterized by the propert that for any $w\in {^IW}$ there exists a $y\in W_I$ such that $\sigma(w)=y w \psi(y)^{-1}$.  Then \cite[Thm. 11.2 \& 11.5]{AZD} state that:

\begin{theorem}\label{thmWJ} The map $W^J \longrightarrow\left\{ E_{\mathcal{Z}}-\hbox{{\rm orbits in }} H_k\right\}$,  given by $ w \mapsto H_w:=E_{\mathcal{Z}} \cdot (\widetilde{w} g^{-1})$, is a bijection. Morever,

\begin{itemize}

\item[1.] for any $w\in {^IW}$ we have $H_w=H_{\sigma(w)}$. In particular,   for any $u\in {W^J}$ the corresponding orbit $H_u$  is a locally closed, smooth subvariety of $H_k$ of dimension $\dim P+\ell(u)$.

\item[2.] for any $w$, $w'\in {W^J}$ we have that $H_{w'} \subset \overline{H}_w$, and we write $w' \preccurlyeq w$, if and only if there exists $y\in W_I$ such that $y w' \psi(y)^{-1} \leq w $ for the Bruhat order $\leq$ on $W$.

\end{itemize}

\end{theorem}

In fact, we have a richer structure of quotient stacks; see for example \cite[\S 2.2]{Fzips}. Recall that if $X$ is a $k$-scheme and $Q$ is a linear algebaric $k$-group scheme acting on the left on $X$ the quotient stack $[Q\backslash X]$ is the stack fibered over the category of $k$-schemes whose $S$-valued points consist of pairs $(I,\alpha)$ with (1) $I$  a left $Q$-torsor over $S$ (2) $\alpha\colon I \to X$ ia  morphism commuting with the left $Q$-actions on $I$ and on $X$. There is a notion of topological space asociated to $[Q\backslash X]$ that we denote $\vert [Q\backslash X]\vert$

If the number of orbits  $Q\backslash X$ in $X$ for the action of $H$ is finite, then such space has a unique topology compatible with the closure relation among orbits; see \cite[Prop. 2.1]{Fzips}. Furthermore, \cite[Prop. 2.2]{Fzips} guarantees that the natural map $X\to [Q\backslash X]$ induces an homeomorphism $(Q(k)\backslash X(k))\cong \vert [Q\backslash X]\vert$. So Theorems \ref{thmIW} and  \ref{thmWJ} describe the topological space underlying the quotient stack $[ E_{\mathcal{Z}}\backslash  H_k ]$.

\subsection{Fine Deligne-Lusztig varieties}\label{sec:DL}

We use the notation of the previous section. We start with a connected reductive group $H$ over $\FF_p$ and a cocharacter $\chi\colon \GG_{m,k}\to H_k$, over $k=\overline{\FF}_p$.
This defines a parabolic subgroup $P\subset H_k$ that we assume to contain a Borel subgroup $B_k$, defined over $\FF_p$.

Let $Q:=P^{(p)}$ be the base change of $P\subset H_k$ via Frobenius on $k$. It is a parabolic subgroup of $(H_k)^{(p)}=H_k$ (as $H$ is defined over $\FF_p$). It contains the Borel subgroup $(B_k)^{(p)}=B_k$ (as $B_k$ is also defined over $\FF_p$). The Frobenius map $\varphi\colon P \to P^{(p)}=Q$ defines an isogeny $\varphi\colon P/{\rm R}_u P \to Q/{\rm R}_u Q$. 

Then $\mathcal{Z}:=\mathcal{Z}_{\rm DL}:=(P,Q,\varphi)$ is a zip datum as in Definition \ref {AZD}. Note that $Q/{\rm R}_u Q\cong (P/{\rm R}_u P)^{(p)}$ and then $\varphi$ is Frobenius.  In particular, the differential of $\varphi$ at $1$ is zero and $\mathcal{Z}_{\rm DL}$ is orbitally finite.

Write $ \mathrm{F}\ell_{H_k,\chi}:=P\backslash H_k$. The projection map $\rho\colon H_k\to  \mathrm{F}\ell_{H_k,\chi}  $ is a $P$-torsor (for the \'etale topology) so that $ \mathrm{F}\ell_{H_k,\chi}$ represents the quotient stack $[P\backslash H_k]$. The  inverse of Lang isogeny  $\gamma\colon H_k \to H_k$, $h\mapsto h \varphi(h)^{-1}$ defines a map of quotient stack $$\overline{\gamma}\colon [P\backslash H_k]\longrightarrow [E_{\mathcal{Z}}\backslash H_k].$$

\begin{proposition}\label{prop:gammabar} The map $\overline{\gamma}$ is smooth. In particular, the induced map on underlying topological spaces is surjective, continuous and open. 

\end{proposition}
\begin{proof} Recall that to give a morphism $S \to  [E_{\mathcal{Z}}\backslash H_k]$ is equivalent to give an $E_{\mathcal{Z}}$-torsor $Y$ over $S$ and a $E_{\mathcal{Z}}$-equivariant morphism $Y\to H_k$. In our case, using that $ \mathrm{F}\ell_{H_k,\chi}$ represents  $[P\backslash H_k]$, the morphism $\overline{\gamma}$ is defined by 
\smallskip

(1) the $E_{\mathcal{Z}}$-torsor $E_{\mathcal{Z}}\times^P H_k$ over  $\mathrm{F}\ell_{H_k,\chi} $ defined by pushing-forward $P\backslash \bigl(E_{\mathcal{Z}}\times_k H_k\bigr)$ of  the inclusion $P\subset H_k$ via the group homomorphism $P\hookrightarrow E_{\mathcal{Z}}$, $ h\mapsto (h,\varphi(h)\bigr)$;

\smallskip
(2) the  $E_{\mathcal{Z}}$-equivariant map  $\nu\colon E_{\mathcal{Z}}\times^P H_k \to H_k$ defined by $\bigl((p,q),h)\mapsto (p,q)\cdot \gamma(h)=p \gamma(h) q^{-1}$. 
\smallskip

Thus the diagram 

$$\begin{matrix}
E_{\mathcal{Z}}\times^P H_k &\stackrel{\nu}{\longrightarrow} & H_k \cr
\downarrow & & \downarrow\cr
 \mathrm{F}\ell_{H_k,\chi} &\stackrel{\overline{\gamma}}{\longrightarrow} &  [E_{\mathcal{Z}}\backslash H_k], \cr
\end{matrix}$$
where the vertical maps are the natural projection maps, is cartesian and in order to prove the claim it sufffices to show that $\nu$ is smooth. The scheme  $E_{\mathcal{Z}}\times^P H_k$ is a smooth over $k$, being a tosor under the smooth algebraic group $E_{\mathcal{Z}}$ over the smooth scheme $\mathrm{F}\ell_{H_k,\chi}$. As also $H_k$ is smooth, in order to prove that $\nu$ is smooth it suffices to show that it is surjective on the tangent space at every $k$-valued point. 

The map $\gamma$ is finite and \'etale so that inclusion $ H_k \subset E_{\mathcal{Z}}\times^P H_k \to  H_k$ is already surjective on tangent spaces and the claim follows. 
\end{proof}

Recall from Theorems \ref{thmIW} and \ref{thmWJ} that $\vert [E_{\mathcal{Z}}\backslash H_k] \vert$ is homeomorphic to ${^I}W$ (resp.~$W^J$), where $W$ is the Weyl group of $H_k$, $I$ is the type fo the parabolic $P$, $J$ is the type of $Q$, we can take $g={\rm Id}$ in loc.~cit.~and the topologial structure on ${^I}W$ (resp.~$W^J$) is defined by the order relation $w' \preccurlyeq w$ described as follows. 

\smallskip

Let $I$ be the type of $P$ and let $J=\varphi(I)$ be the type of $Q=P^{(p)}$. Frobenius on $H_k$ defines an isomorphism of Weyl groups $\varphi\colon W\to W$ and a bijection on simple reflections, as the Weyl group and the simple reflections are defined with respect to a Borel subgroup and a torus defined over $\FF_p$. Then $\varphi$ defines an isomorphism 
\begin{equation}\label{foulapsiDL} \psi_{\rm DL}\colon W_I \longrightarrow W_J, \quad w\mapsto \varphi(w)\end{equation}  
and, thanks to Theorems \ref{thmIW} and \ref{thmWJ}, for $w$ and $w'\in {^I W}$ (or in $W^J$) we have $w' \preccurlyeq w$ iff there exists $y\in W_I$ such that $y w' \psi_{\rm DL}(y)^{-1} \leq w$.

\begin{definition}  For every $w\in {^I} W $ we let $\mathrm{F}\ell_{H_k,\chi,w}$ be the inverse image of $w\in {^I}W$ via the map of topologial space  $\mathrm{F}\ell_{H_k,\chi}\to \vert [E_{\mathcal{Z}}\backslash H_k] \vert \cong {^I}W$ with reduced induced scheme structure.

Similarly for $u\in W^J$ we let $\mathrm{F}\ell_{H_k,\chi}^u$ be the inverse image of $u\in W^J$ via the map of topologial space  $\mathrm{F}\ell_{H_k,\chi}\to \vert [E_{\mathcal{Z}}\backslash H_k] \vert \cong W^J$.
\end{definition}

By theorem \ref{thmWJ} there is a bijection $\sigma\colon {^I W}\cong W^J$ and for every $w\in {^I}W$ we have  $\mathrm{F}\ell_{H_k,\chi,w}=\mathrm{F}\ell_{H_k,\chi}^{\sigma(w)}$.  

\begin{corollary} Each $\mathrm{F}\ell_{H_k,\chi,w}$ and each $\mathrm{F}\ell_{H_k,\chi}^u$ is smooth and equidimensional of dimension $\ell(w)$, resp.~$\ell(u)$. Moreover  $\mathrm{F}\ell_{H_k,\chi,w'}$ is in the closure of $ \mathrm{F}\ell_{H_k,\chi,w}$ (respectively $\mathrm{F}\ell_{H_k,\chi}^{u'}$ is in the closure of $\mathrm{F}\ell_{H_k,\chi}^u$) if and only if  $w' \preccurlyeq w$ (respectively $u'\preccurlyeq  u$).

\end{corollary}
\begin{proof} We provide a proof for the strata $\mathrm{F}\ell_{H_k,\chi,w}$. For the strata $\mathrm{F}\ell_{H_k,\chi}^u$ the argument is the same. The closure relation and  the fact that  each $\mathrm{F}\ell_{H_k,\chi,w}$ is non-empty  follows from the fact that the map of topological spaces induced by $\overline{\gamma}$ is continuous, open and surjective.  This also implies that such map respects codimensions. Since by Theorem \ref{thmIW} the codimension of $w\in {^I}W$ is  $\dim H_k-(\dim P+\ell(w))$ it follows that the codimension of $\mathrm{F}\ell_{H_k,\chi,w}$ in $\mathrm{F}\ell_{H_k,\chi}$ is also $\dim H_k-(\dim P+\ell(w))$. Hence its dimension is $\ell(w)$ as wanted. The smoothness assertion follows from the smoothness of $\overline{\gamma}$. 

\end{proof}

In particular, we get a stratification on $\mathrm{F}\ell_{H_k,\chi}$ with strata $\mathrm{F}\ell_{H_k,\chi,w}$, for $w\in {^I}W$, called {\it fine Deligne-Lusztig varieties}, introduced by He \cite{He}.

\subsubsection{A variant}\label{sec:variant} We consider the  the flag variety $\mathrm{F}\ell_{H_k,\chi^ {-1}}:=P_-\backslash H_\kappa$.  We describe an isomorphism $$ \mathrm{F}\ell_{H_k,\chi^ {-1}} \cong  \mathrm{F}\ell_{H_k,{^{w_0}}\chi^{-1}},$$ where $w_0\in W$ the longest element in the Weyl group $W$, compatible with stratifications. Later in the paper we will be lead to work with the first space but we will need the dictionary provided by this isomorphism to get a description of the stratification in terms of the Weyl group $W$ and the fixed set of simple reflections $\Delta$.

Recall that $I$ is the type of $P$ and $I_-$ is the type of the parabolic $P_-$ defined by the cocharacter $\chi^{-1}$. As $P_-$ is the parabolic opposite to $P$ we have $P_{I_-}={^{w_0}}P_-$ so that $I_-={^{w_0} I}$.  Recall also that Frobenius defines an automorphism $\varphi$ of the Weyl group $W$ and that $J:=\varphi(I)$ is the type of $P^{(p)}$ and $K=\varphi(I_-)={^{w_0}J}$ is the type of $Q=P_-^{(p)}$. Consider the elements $y:={^I}w_0=w_0^{I_-}$ and $z=\varphi(y)={^J}w_0=w_0^K$ of $W$ defined in Lemma \ref{lemma:z}. Then $P_-={^y} P_{I_-} $ and conjugation by $\widetilde{y}^{-1}$ provides the identification $L \cong {^{w_0}} L$ of the Levi subgroup $L=P\cap P_-$ of $P_-$ and the Levi subgroup of $P_{I_-}$ so that $B\cap L \cong B\cap {^{w_0}} L$. Similarly,   conjugation by $\widetilde{z}^{-1}$ provides the identification $L^{(p)} \cong {^{w_0}} L^{(p)}$ of the Levi subgroup $Q=P_-^{(p)}$ and the Levi subgroup of $P_K=P_{I_-}^{(p)}$ so that $B\cap L^{(p)} \cong B\cap  {^{w_0}} L^{(p)}$. Consider the zip data $\mathcal{Z}_{\rm opp}:=(P_-,Q,\varphi)$ and $\mathcal{Z}_-:=(P_{I_-},P_K,\varphi)$.

\begin{lemma}\label{lemma:EZEZ'} The map $ E_{\mathcal{Z}_{\rm opp}}\longrightarrow  E_{\mathcal{Z}_-}$, $(p,q)\mapsto \bigl(\widetilde{y}^{-1}p \widetilde{y},\widetilde{z}^{-1}q \widetilde{z}\bigr)$, defines an isomorphism of algebraic groups.

Similarly, the map $H_\kappa\longrightarrow H_\kappa$, $h\mapsto \widetilde{y}^{-1} h  \widetilde{z}$, defines an isomorphism which is equivariant for the action of  $ E_{\mathcal{Z}_{\rm opp}}$ and $  E_{\mathcal{Z}_-}$ respectively and an isomorphism  $ \mathrm{F}\ell_{H_k,\chi^ {-1}} \longrightarrow  \mathrm{F}\ell_{H_k,{^{w_0}}\chi^{-1}}$ making the following diagram (see Proposition \ref{prop:gammabar})

$$\begin{matrix} \mathrm{F}\ell_{H_k,\mu^ {-1}}  & \longrightarrow & [E_{\mathcal{Z}_{\rm opp} }\backslash H_k]\cr
 \big\downarrow & & \big\downarrow \cr
 \mathrm{F}\ell_{H_k,{^{w_0}}\mu^{-1}} & \longrightarrow & [E_{\mathcal{Z}_-}\backslash H_k]
\end{matrix}$$commute. In particular, the topological space underlying $[E_{\mathcal{Z}_-}\backslash H_k]$ admits a description in terms of ${^{I_-}}W$ with the order  $\preccurlyeq$ as in Theorem \ref{thmIW} and the scheme $ \mathrm{F}\ell_{H_k,\mu^ {-1}} $ admits an induced startification labeled by ${^{I_-}}W$.
\end{lemma}

\subsection{$H$-Zips}\label{sec:Zips}

Recall the notation. We have a connected reductive group $H$ over $\FF_p$ and a cocharacter $\chi\colon \GG_{m,k}\to H_k$, over $k=\overline{\FF}_p$, that  defines a parabolic subgroup $P=P_+\subset H_k$. We assume it contains a Borel subgroup $B_k$, defined over $\FF_p$.

Denote by $P_-$ be the associated opposite parabolic subgroup of $H_k$ defined by the cocharater $\chi^{-1}$. It is characterized by the property that $\Lie P_-\subset \Lie H_k$ is the sum of the non-positive weight spaces for the adjoint action of $\GG_{m,k}$ via $\chi$. Let $U$ and $V$ be the unipotent radicals of $P$ and $P_-$ respectively.  The quotients $P/U$ and $P_-/V$ are isomorphic reductive groups and we denote such  algebraic group by $L$. Set $Q:=P_-^{(p)}\subset H_k^{(p)}=H_k$. It is a parabolic subgroup of $H_k$ with unipotent radical $V^{(p)}$. The Frobenius map $\varphi\colon P \to P^{(p)}$ defines an isogeny $\varphi\colon P/U\cong L \to L^{(p)} \cong Q/V^{(p)}$. 

Also in this case $\mathcal{Z}:=\mathcal{Z}_{\rm EO}:=(P,Q,\varphi)$ is a zip datum as in Definition \ref {AZD}. It is orbitally finite as $\varphi$ has trivial differential at $1$. We consider the  quotient stack $\bigl[E_{\mathcal{Z}_{\rm EO}} \backslash H_k\bigr]$.
Following \cite[Def. 1.4]{Fzips} we give the following:

\begin{definition}\label{def:HZIP} We define the  category $H-\hbox{\rm Zip}^\chi$ 
fibred over the category $\Sch_k$ of schemes over $k$  as follows.\smallskip

For a scheme $S$ over $k$ the objects  are quadruples $\bigl(I,I_+,I_-,\iota\bigr)$ where $I$ is a right $H_k$-torsor, $I_+\subset I$ is a right $P$-torsor,
$I_-\subset I$ is a right $Q$-torsor and $\iota\colon I_+^{(p)}/U^{(p)} \to I_-/V^{(p)}$ is an isomorphism of $L^{(p)}$-torsors.\smallskip

A morphism of two objects $\bigl(I,I_+,I_-,\iota\bigr))\to \bigl(I',I_+',I_-',\iota'\bigr) $  is a triple of equivariant morphisms $I\to I'$, $I_{\pm} \to
I_{\pm}'$, which are compatible with the inclusions $I_{\pm}\subset I$ and $I_{\pm}'\subset I'$ and and with $\iota$ and $\iota'$.\smallskip

\end{definition}

It is easily verified that  $H-\hbox{\rm Zip}^\chi$  is a stack over  $\Sch_k$.  There is a morphism $H_k\to  H-\hbox{\rm Zip}^\chi$, given by \cite[Construction 3.4]{Fzips}. To a $k$-scheme $S$ and an element $g\in H_k(S)$ we associate 
$I_g=S\times_k H_k$, $I_{g,+}:=S\times_k P$, $I_{g,-}:=g\cdot (S\times_k Q)\subset S\times_k H_k$ , $$\iota_g\colon S\times_k P/U=S\times_k L \stackrel{1\times \varphi}{\longrightarrow} S\times_k L^{(p)}= (S\times_k Q/V^{(p)}) \stackrel{g\cdot \_}{\longrightarrow} 
I_{-,g}/V^{(p)}.$$Then $I_{\underline{g}}:=\bigl(I_g,I_{g,+},I_{g,-},\iota_g\bigr)\in H-\hbox{\rm Zip}^\chi(S) $. By
\cite[Prop. 3.11]{Fzips}:

\begin{proposition}\label{prop:Ig} The map $H_k\to H-\hbox{\rm Zip}^\chi$,  $g\mapsto I_{\underline{g}}$, induces an isomorphism of stacks  $\bigl[E_{\mathcal{Z}_{\rm EO}} \backslash H_k\bigr] \cong H-\hbox{\rm Zip}^\chi$. 
\end{proposition}

Let $I$ be the type of $P$, let $J=\varphi(I)$ be the type of $P^{(p)}$ and let $K={^{w_0}J}$ be the type of $Q$ which is the parabolic opposite to $P^{(p)}$. Here we use that Frobenius on $H_k$ defines an isomorphism of Weyl groups $\varphi\colon W\to W$ and a bijection on simple reflections, as the Weyl group and the simple reflections are defined with respect to a Borel and a torus defined over $\FF_p$. In particular, by Theorem \ref{thmIW} and Theorem \ref{thmWJ}, we have

\begin{corollary}\label{cor:HZipIW} The map  ${^I}W \to H-\hbox{\rm Zip}^\chi$, $w\mapsto [\widetilde{w} \widetilde{z}^{-1}]$ with $z=w_0^K$,  defines a homeomorphism of underlying topological spaces, and  equivalently for $W^K$ instead of ${^I}W $, with the topological structure defined by the partial order $ \preccurlyeq$.

\end{corollary}

We can describe the order relation as follows.  In the notation of Theorem  \ref{thmIW}, we have $g=\widetilde{z}$ with $z=w_0^K$ and $\psi$ is the map 
\begin{equation}\label{formulapsizip} \psi_{\rm EO}\colon (W_I,I)\cong (W_K,K), \qquad w\mapsto z\varphi(w) z^{-1}.\end{equation} 
Thanks to Theorems \ref{thmIW} and \ref{thmWJ},  for $w$ and $w'\in {^I W}$ or $w$ and $w'\in W^K$, we have $w' \preccurlyeq w$ iff there exists $y\in W_I$ such that $y w' \psi_{\rm EO}(y)^{-1} \leq w$.

\subsection{The Ekedahl-Oort stratification}\label{sec:EO}

Recall from  the introduction that we have a scheme $S=S_K$, the integral canonical model  of the Shimura variety  $\mathrm{Sh}_K(G,X)$ over $\cO_{E,v}$, and its sepcial fiber $S_0$ over the   residue field $\kappa$ of $\cO_{E,v}$.

By \cite[\S 2.3.1]{Kisin} we may assume that the inclusion $G\subset {\rm GSp}_{2g}$ extends to a morphism of algebraic groups $G_{\ZZ_p} \subset {\rm GSp}_{2g,\ZZ_p}$ where  $G_{\ZZ_p}$ is the given model of $G$ over $\ZZ_p$ and  ${\rm GSp}_{2g,\ZZ_p}$ is the base change from $\ZZ_{(p)}\to \ZZ_p$ of the group of symplectic similitudes of free $\ZZ_{(p)}$-module $\Lambda$ of rank $2g$ endowed with a symplectic perfect pairing $\psi$. The subgroup $G_{\ZZ_p}$ can be realized as the subgroup fixing a finte set  of elements $(s_\alpha)_\alpha$ in the tensor algebra of $\Lambda$.

Furthermore,  possibly after refining the prime-to-$p$ level subgroup $K^p\subset G(\AA^{f,(p)}$ we may assume that there exists an hyperspecial level $K_g:=K_{g,p} K_g^p\subset {\rm GSp}_{2g,\ZZ_p}(\AA^f)$ inducing a morphism of Shimura varieties $\mathrm{Sh}_K(G,X)\to \mathrm{Sh}_{K_g}\bigl({\rm GSp}_{2g},\mathbb{H}_g\bigr)$. This provides  a finite morphism of the integral models $$\xi\colon S_K \longrightarrow S_{K_g} \otimes_{\ZZ_{(p)}} \cO_{E,v},$$where $S_{K_g} $ is a Siegel modular variety. In particular, we may assume that it is a fine moduli space so that it carries a universal abelian scheme.

We set $H:=(G_{\ZZ_p})_{\FF_p}$.  As explained in \cite[\S 4.1.7]{GK} there exists a Borel subgroup $B_{\ZZ_p}\subset G_{\ZZ_p}$. The Hodge cocharacter admits a representative $\widetilde{\mu}$ in its conjugacy class defined over $\widehat{\cO}_{E,v}$ ($\cong \WW(\kappa)$).  Let $\chi\colon \GG_{m,\kappa}\to H_\kappa$ be the reduction of $\widetilde{\mu}$ modulo $p$. The composite of $\widetilde{\mu}$ with the inclusion $G_{\WW(\kappa)} \subset {\rm GSp}_{2g,\WW(\kappa)}$ defines a splitting $\Lambda\otimes \WW(\kappa)=\Lambda_0 \oplus \Lambda_{-1}$ where $\GG_m$ acts on $\Lambda_i$  via the character $z\mapsto z^{-i}$ on $\Lambda_{-i}$ for $i=0$, $-1$.

The main result of \cite[Thm. 3.1.2]{Chao} provides a smooth map $$\zeta\colon  S_0\longrightarrow H-\hbox{{\rm Zip}}^\chi.$$This is refined in \cite[Thm.~6.2.1]{GK}, using the work \cite{MP} on toridal compactifications,  to a morphism $$\zeta^\Sigma \colon  S_0^\Sigma\longrightarrow H-\hbox{\rm Zip}^\chi,$$where $S_0^\Sigma$ is the special fiber of the smooth and projective toroidal compactification  $S^\Sigma$ of $S$ associated to a finite and admissible rational cone decomposition $\Sigma$. We make the following Assumption, following \cite{GK}, in order to define a well behaved compactification of the universal abelian scheme over $S_{K_g}^{\Sigma_g}$ and hence over $S^\Sigma$:

\begin{assumption}\label{ass} $\Sigma$  is supposed to refine the pull-back via $\xi$ of a  finite and admissible  rational cone decomposition $\Sigma_g$ for the Siegel case such that $\Sigma_g$ is log-integral (see \cite[\S 6.1.2]{GK}.  
\end{assumption}

We denote by $\xi^\Sigma\colon S^\Sigma\to S_{K_g}^{\Sigma_g}$ the induced morphism on toroidal compactifications.

\

For the rest of this section we  recall the construction of the map $\zeta^\Sigma$. Let $\cS^{\Sigma}$ be the $p$-adic formal scheme defined by completing $S^{\Sigma}$ along its special fiber $S_{0}^{\Sigma}$. The first de Rham cohomology of the universal  abelian scheme over $S$, obtained by pull-back from $S_{K_g}$, extends to a locally free $\cO_{S^{\Sigma}}$-module ${\rm H}_{\rm dR}^1$ of rank $2g$ endowed with

\smallskip

i)  the Hodge filtration $\Omega\subset {\rm H}_{\rm dR}^1$ which is a locally free sheaf of rank $g$, with locally free quotient;

\smallskip

ii) a connection $\nabla$  with log poles along the complement of $S$ in $S^{\Sigma}$, extending the Gauss-Manin connection.
The  base change $\cH$ of ${\rm H}_{\rm dR}^1$ to $\cS^{\Sigma}$, together with the connection, is the evaluation along the DP-thickening $S_0^\Sigma\subset \cS^{\Sigma} $ of the logarythmic crystal over $S_0^\Sigma$ defines as the dual of the covariant log Dieudonn\'e module of the mod $p$ reduction of the degenerating universal abelian scheme over $S_0^\Sigma$ (see   \cite[Prop. 1.3.5]{MP} and the discussion in \cite[\S 4.3.1]{MP});

\smallskip

iii) if ${\rm Spf}(A)\subset \cS^{\Sigma}$ is an affine open and we choose a lift $\varphi\colon A\to A$ of Frobenius on $A_0=A/pA$, then the evaluation $\cH_A$ of the crystal $\cH$ at $A$ is endowed with a $\varphi$-linear map, the Frobenius, $$F\colon \cH_A^{(p)} \longrightarrow \cH_A.$$

\smallskip

iv)  elements $(s_{\rm dR,\alpha})_\alpha$ in the tensor algebra of ${\rm H}_{\rm dR}^1$ such that $$J:={\rm Isom}\left(\bigl(\bigl(\Lambda , (s_\alpha)_\alpha\bigr)\otimes \cO_{S^\Sigma}, {\rm H}_{\rm dR}^1,(s_{\rm dR,\alpha})_\alpha\bigr)\right)$$is a right $G_{\ZZ_p}$-torsor over $S^\Sigma$. If we further consider $g\in J$ such that $g$ sends $\Omega$ isomorphically onto $\Lambda_0\otimes \cO_{S^\Sigma}$ we have a torsor $J_+$ under the parabolic $P^+_{\widetilde{\mu}}$  defined by $\widetilde{\mu}$. 

\

Choose $A$ and $\varphi$ as in (iii).  Then the  restriction of $F$  to $\Omega$ is divisible by $p$ and it is proven in \cite[Lemma 6.3.4]{GK} (see especially the first paragraph on page $938$)  that $\frac{F}{p}(\Omega) + {\rm Im} F=\cH_A$. 
Take $A$ small enough so that $J_+(A)$ is non-empty and choose $t\in J_+(A)$.  We deduce that the $A$-linear map $$g_t:=t^{-1}\circ \left(\frac{F}{p}\vert_{t^{(p)}(\Lambda_0\otimes A)} \oplus F\vert_{t^{(p)}(\Lambda_{-1}\otimes A)}\right) \circ t^{(p)}  \colon \Lambda\otimes A \to \Lambda \otimes A$$defines an element of $G(A)$. Let $g_{t,0}\in G(A_0)$ be its reduction modulo $p$.  

\

{\it Definition of the $H$-zip $\zeta^\Sigma(A_0)$}.\enspace  The element $g^{-1} t=(t^{-1} g)^{-1}$ defines an isomorphism  $\cH_A \cong  \Lambda\otimes A$. Denote by $J_-$ the $P_-^{(p)}$-torsor of all its trivializations. Let $I$, $I_+$, $I_-$ be the reduction modulo $p$ of $J$, $J_+$ and $J_-$. They are right $H_k$, resp.~$P$, resp.~$Q$-torsors. The map  $\gamma:=t \circ g_t \circ  \bigl(t^{(p)}\bigr)^{-1} = \left(\frac{F}{p}\vert_{t^{(p)}(\Lambda_0\otimes A)} \oplus F\vert_{t^{(p)}(\Lambda_{-1}\otimes A)}\right)\colon \cH_A^{(p)} \to \cH_A$ defines an isomorphism, as explained above. Let $\gamma_0$ be its reduction modulo $p$. Composing with $\gamma_0$ we get  an isomorphism $\iota\colon I_+^{(p)}/U^{(p)} \to I^-/ V^{(p)}$ of $L^{(p)}$-torsors. Then, $(I,I_+,I_-,\iota)$ defines an $H$-zip in the sense of Dedinition \ref{def:HZIP} and $$\zeta^\Sigma(A_0):=(I,I_+,I_-,\iota).$$

Then, in the notation of Proposition \ref{prop:Ig} (cf. \cite[Rmk. 3.2.8]{Chao}  and  \cite[Prop. 6.7]{Wortman}),

\begin{lemma}\label{lemma:zetaSigma}  We have $\zeta^\Sigma(A_0)=I_{\underline{g}_{t,0}}$. 

Furthemore, if  $z\in \widetilde{P}(A)$,  we consider the element $ t \circ z   \in J_+(A)$ and we write $z=\ell u$ with $\ell\in \widetilde{L}(A)$ and $u\in {\rm R}_u \widetilde{P}(A)$ then  $$g_{t \circ z ,0} \equiv z_0^{-1}  g_{t,0} \ell_0^{(p)} .$$

\end{lemma}
\begin{proof} The first claim is a direct verification. The last formula follows as  Frobenius is divisible by $p$ on $t^{(p)}(\Lambda_0\otimes A)$.

\end{proof}

\section{Smoothness of $\zeta^\Sigma$}

As conjectured by Wedhorn and Ziegler in \cite[Rmk. 6.4]{WZ}, we have:

\begin{theorem}\label{theorem:zetasigmasmooth} The map $\zeta^\Sigma$ is smooth. In particular, it is open. 

\end{theorem}
The rest of this section is devoted to the proof of the Theorem and of some of its consequences.  By construction $ S_0^\Sigma$ is a smooth scheme and $H-\hbox{\rm Zip}^\chi$ is a smooth stack by \cite[Cor. 3.12]{Fzips}. Hence, it suffices to prove that $\zeta^\Sigma$ is surjective on tangent spaces at every closed point $x_0\in S_0^\Sigma$. If $x_0$ lies in $S_0$ this is  the content of \cite[Thm. 4.1.2]{Chao}. Assume that $x_0\in S_0^\Sigma\backslash S_0$. Let $A$ be the complete local ring of $S^\Sigma$ at $x_0$. Denote by $\cH_A$ the Frobenius logarithmic crystal relative to the DP-Thickening $A_0=A/pA\subset A$ obtained by pull-back to $A$ of the logarithmic crystal $\cH$ relative to the DP-thickening $S_0^\Sigma\subset \cS^\Sigma$ defined in \S \ref{sec:EO}. 

\subsection{The structure at the boundary}\label{sec:boundary} As explained in \cite[Thm.  4.1.5]{MP} there exist  

\begin{itemize}
\item[i.]  a smooth integral model  $S'$ over $\WW(\kappa)$ associated to a Shimura datum $(G',X')$ of Hodge type  with hyperspecial level $K'$;
\item[ii.] a torsor $\overline{S}' \to S'$  under an abelian scheme $\mathcal{A}$ over$S'$;
\item[iii.]   a torsor $\widetilde{S} \to \overline{S}'$  under a torus $\mathbf{E}$ over $\WW(\kappa)$;
\item[iv] a smooth torus embedding $\widetilde{S}\subset \widetilde{S}^\Sigma$ relative to $\overline{S}'$ 
\end{itemize}

such that the complete local ring $A$ of $S^\Sigma$ at $x_0$ is isomorphic to the completion of $\widetilde{S}^\Sigma $ at a closed point $s_0$. 

\smallskip

Over $S'$ one has a universal principally polarized abelian scheme $B$. Define the open $U\subset {\rm Spec}(A)$ as the inverse image of $\widetilde{S} \subset \widetilde{S}^\Sigma$. By loc. cit. one has a degenerating $1$-motive $M_U$.  Write $M_U^\vee=[Y \to D_U]$ for the Cartier dual $1$-motive; then $D$ is a semiabelian scheme over $\overline{S}'$, extension of the principally polarized abelian scheme $B$ over $S'$ by a torus $T$ which has constant rank.   By the discussion in \cite[\S 4.3.1]{MP} the Frobenius logarithmic crystal $\cH_A$ is identified with the contravariant log Dieudonn\'e crystal of the mod $p$ reduction of $M_U$.

In particular, as an $A$-module $\cH_A$  is endowed with a logarithmic connecton $\nabla$, with a two step Hodge filtration $$0=F^2 \cH_A \subset F^1 \cH_A \subset F^0\cH_A=\cH_A$$ and a three step weight filtration $$ 0 \subset W_0 \cH_A \subset W_{1} \cH_A \subset W_{2} \cH_A=\cH_A$$consisting of Frobenius subcrystals such that
\begin{itemize}

\item[a.]   $W_{1} \cH_A$ is identified with the covariant de Rham realization \cite[\S 1.1.3]{MP} of $D$ over $A$, i.e.,  the Lie algebra of the universal extension of the  pull back of $D$ to  $A$ and the Frobenius is obtained from its identification with the covariant crystalline realization of the base change of $D$ to $A_0=A/pA$; 

\item[b.]  $F^1 \cH_A$ is identified with the non-trivial step of the Hodge filtration of the covariant de Rham realization of $D$ over $A$;

\item[c.]  $W_0 \cH_A$  is identified with the  Lie algebra of the toric part $T$ of $D$ and the Frobenius structure is defined by the identification with the covariant Dieudonn\'e module of the base change of $T$ to $A_0$. 

\item[d.]  the $A$ module ${\rm Gr}_{1}^W \cH_A$ is identified with de Rham homology of $B$ and its Frobenius structure arises from the identification with the covariant Dieudonn\'e module of its base change  $B_0$ to $A_0$; 

\item[e.]  ${\rm Gr}_{2}^W \cH_A$ is identified with with the  Lie algebra of the universal extension of $Y$ and the Frobenius structure is defined by the identification with the covariant Dieudonn\'e module of the \'etale $p$-divisible group $Y\otimes \QQ_p$ over $A_0$.

\end{itemize}

\subsection{The underlying algebraic groups}

We summarize the algebraic group counterpart of the description in \S \ref{sec:boundary}; see the discussion in \cite[\S 4.3.4]{MP}. The algebraic group $G'=N/U$, appearing in the Shimura datum $(G',X')$,  is the quotient of a normal subgroup $N$ of a  parabolic subgroup $R$ of $G$ by the unipotent radical $U$ of $N$ and $U$ is also the unipotent radical of $R$:

$$U\subset N \subset R \subset G.$$All these groups are defined over $\QQ$ and they extend to subgroups $$U_{\ZZ_p} \subset N_{\ZZ_p} \subset R_{\ZZ_p}\subset G_{\ZZ_p}$$with $R_{\ZZ_p}$  a parabolic subgroup  of the model $G_{\ZZ_p}$ of $G$ over $\ZZ_p$, $N_{\ZZ_p}$ a normal subgroup of $R_{\ZZ_p}$ and $U_{\ZZ_p} $ the unipotent radical of $R_{\ZZ_p}$ (equivalently of $N_{\ZZ_p}$). The quotient $N_{\ZZ_p}/U_{\ZZ_p} $  is a reductive group $G'_{\ZZ_p}$. It is a normal subgroup of the reductive group $G''_{\ZZ_p}:=R_{\ZZ_p}/U_{\ZZ_p}$ and  defines a model of $G'$ over $\ZZ_p$. In fact, the parabolic $R$ is defined by a cocharacter, called the weight cocharacter, $w\colon \mathbb{G}_{m,\QQ} \to G$ that factors via $N$. By \cite[Prop.~3.2.1 \& \S 3.2.2]{Chao}, after possibly conjugating $w$ by $N$, we may assume that its base change to $\QQ_p$  extends to a cocharacter $w_{\ZZ_p}\colon  \mathbb{G}_{m,\ZZ_p} \to N_{\ZZ_p}$.  Using the embedding $G_{\ZZ_p}\subset   \mathrm{GSp}(\Lambda \otimes \ZZ_p)$ the cocharacter $w_{\ZZ_p}$ defines a grading on  $\Lambda \otimes \ZZ_p$ and, hence, an ascending  filtration $W_\bullet \bigl(\Lambda \otimes \ZZ_p\bigr)$, called the weight filtration: $$\{0\}=W_{-1} \bigl(\Lambda \otimes \ZZ_p\bigr) \subset W_0 \bigl(\Lambda \otimes \ZZ_p\bigr) \subset W_1 \bigl(\Lambda \otimes \ZZ_p\bigr) \subset W_2 \bigl(\Lambda \otimes \ZZ_p\bigr)=\Lambda \otimes \ZZ_p.$$

Thanks to \cite[Prop.~12.1\& Lemma 12.2]{Pink} the Hodge cocharacter $\mu$ of $G$ can be chosen to factor through $N_E$, where $E$ is the reflex field of $\mathrm{Sh}_K(G,X)$, and such that it acts trivially on $U_E$ (via the adjoint action).  

\begin{lemma} There exists a cocharacter $\widetilde{\mu}$ of $G_{\WW(\kappa)}$ in the conjugacy class of $\mu$  that factors through $N_{\WW(\kappa)}$.
\end{lemma}
\begin{proof} First  consider the  cocharacter $\mu'$ of the reductive group $G'_E$ induced by $\mu$. Due to  \cite[Prop.~3.2.1 \& S 3.2.2]{Chao}  there exists a conjugate cocharacter $\widetilde{\mu}'$ that is defined for $G'_{\WW(\kappa)}$. As $N_{\WW(\kappa)} \to G'_{\WW(\kappa}$ has unipotent kernel the cocharacter $\widetilde{\mu}'$ lifts uniquely to the sought for cocharacter $\widetilde{\mu}$.  
\end{proof}

The cocharacter $\widetilde{\mu}$ defines an increasing filtration, called the Hodge filtration, $F^\bullet \bigl(\Lambda \otimes \WW(\kappa)\bigr)$: $$\{0\}=F^2 \bigl(\Lambda \otimes \WW(\kappa)\bigr)\subset F^1 \bigl(\Lambda \otimes \WW(\kappa)\bigr) \subset F^0 \bigl(\Lambda \otimes \WW(\kappa)\bigr)=\Lambda \otimes \WW(\kappa)$$such that $W_{0} \bigl(\Lambda \otimes \ZZ_p\bigr) \otimes_{\ZZ_p} \WW(\kappa)\subset F^1 \bigl(\Lambda \otimes \WW(\kappa)\bigr)\subset W_1 \bigl(\Lambda \otimes \ZZ_p\bigr) \otimes_{\ZZ_p} \WW(\kappa)$. 

We let $\widetilde{P}\subset G_{\WW(\kappa)}$, resp.~$\widetilde{P}'\subset G_{\WW(\kappa)}'$, resp.~$\widetilde{P}''\subset G_{\WW(\kappa)}''$ be the parabolic subgroups defined by the cocharacter $\widetilde{\mu}$ of $G_{\WW(\kappa)}$ and the induced coocharacters $\widetilde{\mu}'$ and $\widetilde{\mu}''$ of $G_{\WW(\kappa)}'$ and $G_{\WW(\kappa)}''$ respectively. The intersection $\overline{P}:=\widetilde{P} \cap R_{\WW(\kappa)}$ is a paraboplic subgroup of $G_{\WW(\kappa)}$; it is the parabolic subgroup containing $U_{\WW(\kappa)}$ and with image $\widetilde{P}''$ in $ G_{\WW(\kappa)}''=R_{\WW(\kappa)}/U_{\WW(\kappa)}$. 

\

Here is the relation between the two descriptions of the boundary, one via degenerating $1$-motives and the other via algebraic groups. Recall that the subgroup $G_{\ZZ_p}$ can be realized as the subgroup fixing a finte set  of elements $(s_\alpha)_\alpha$ in the tensor algebra of $\Lambda$. By \cite[Prop.~4.3.7]{MP} one has canonically defined  tensors $(s_{\alpha,\rm dR})_\alpha$ in the tensor algebra $\cH_A$.

\begin{proposition}\label{prop:reductionoverP}  The set of isomorphisms $\cH_A \cong \Lambda\otimes A$, sending the tensors   $(s_{\alpha,\rm dR})_\alpha$  to the tensors $(s_\alpha \otimes 1)_\alpha$, the weight filtration $W_\bullet \cH_A$ to the weight filtration $W_\bullet \Lambda\otimes A$ and the Hodge filtration $F^\bullet \cH_A$ to the Hodge filtration $F^\bullet (\Lambda\otimes \WW(\kappa))\otimes_{\WW(\kappa)} A$, defines a $\overline{P}$-torsor  for the \'etale topology on ${\rm Spec}(A)$.

\end{proposition}
\begin{proof} If we forget the compatibility with the weight filtrations, this is the content of \cite[Prop.~4.3.9]{MP}. The isomorphisms preserving tensors and Hodge filtrations are proven in loc.~cit.~to be a $\widetilde{P}_A$-torsor $I$ over ${\rm Spec}(A)$. Let $J\subset I$ be the subscheme of isomorphisms preserving also the weight filtration. 
Take the base change of  $\widetilde{S}^\Sigma$ via $\WW(\kappa)\to \WW(k)$ with $k$ an algebraic closure of $\kappa$  and replace $A$ with the complete local ring $\overline{A}$ of this base change at a point over $s_0$. Since $A\to \overline{A}$ is faithfully flat, it suffices to show that $J$ is a $\overline{P}$-torsor over an fpqc cover of ${\rm Spec}(\overline{A})$. In order to simplify the notation we simply write $A$ in place of $\overline{A}$. As now $A$ is complete local with algebraically closed residue field, $I(A)$ is non empty and $ I$ is the trivial $\widetilde{P}$-torsor. It suffices to show that $J(A)$ is non-empty as, then, $J$ is the trivial $\overline{P}$-torsor over ${\rm Spec}(\overline{A})$.

The tensors $(s_{\alpha,\rm dR})_\alpha$  are horizontal for the logarithmic connection $\nabla$ on $\cH_A$. The residue of the connection defines the monodromy operator $N_A$ on $\cH_A$. This is an $A$-linear operator such that $N_A^2=0$,  the kernel is $W_1 \cH_A$ and the image is $W_0 \cH_A$. Given an isomorphism $j\colon \cH_A \cong \Lambda\otimes A$ in $I(A)$ the operator $N:=j \circ N_A \circ j^{-1}$ annihilates the tensors $s_\alpha \otimes 1$ and, hence, lies in ${\rm Lie}\, G(A)$.  The weight filtration $j(W_\bullet \cH_A)$, defined by the image of $W_\bullet \cH_A$ via $j$, coincides with $0\subset {\rm Im}(N) \subset {\rm Ker}(N) \subset \Lambda\otimes A$ and, in particular, it is split by a cocharacter $w_A\colon \mathbb{G}_{m,A}\to G_A$; c.f. the proof of \cite[Lemma 3.5.3]{MP}.

Lift $s_0$ to a $\WW(k)[1/p]$-point $s$ of $\widetilde{S}^\Sigma$  lying in the same boundary component as $s_0$. Let $\widehat{A}_s$ be the complete local ring of  $\widetilde{S}^\Sigma$ at $s$. We claim that $J\vert_{\widehat{A}_s}$ is a torsor over ${\rm Spec}(\widehat{A}_s)$.  Take an embedding $\WW(k)[1/p]\subset \CC$. Let $\widehat{B}_s$ be the complete local ring of $\widetilde{S}^\Sigma\otimes_{\WW(\kappa)} \CC$ at the $\CC$-valued point associated to $s$. By construction  we have a variation of mixed Hodge structures over  $\widetilde{S}^\Sigma\otimes_{\WW(\kappa)} \CC$ whose de Rham realization over $\widehat{B}_s$ is the base change of $\cH_A$ and whose underlying $\QQ$-vector space is $\Lambda \otimes \QQ$ so that $J(\widehat{B}_s)$ is non-empty; see the discussion in \cite[\S 2.1]{MP}.   The extension $ \widehat{A}_s \to \widehat{B}_s$ is faithfully flat as it  is local and $ \widehat{B}_s$ is obtained from $ \widehat{A}_s$ by a flat base change $\WW(k)\to \CC$, localization and completion. This proves that the cocharacters $w_A$ and $w$ of $G$ are conjugate over $A_s$. It follows from \cite[Prop. 3.2.1]{Chao} that the conjugacy classes of cocharacters of $G_A$ are the disjoint union of copies of ${\rm Spec}(A)$. We deduce that $w_A$ and $w$ are conjugate over an fpqc cover of $A$, i.e., the filtrations  $j(W_\bullet \cH_A)$  and $W_\bullet \Lambda$ are conjugate under $G_A$ over an fpqc cover of $A$. Over this $J$ becomes a $\overline{P}$-torsor as wanted.

\end{proof}

\subsection{The proof}

Consider the parabolic subgroups $\overline{P}_\kappa \subset P \subset H=G_\kappa$ which are the reduction of $\overline{P} \subset\widetilde{P}\subset G_{\WW(\kappa)}$. Fix a Borel subgroup $B\subset \overline{P}_\kappa$ and a maximal torus $T$. Let $H'':=G''_{\FF_p}$, resp.~$H'=G'_{\FF_p}$ be the reductive quotient of   the mod $p$-reduction $R_{\FF_p}$ of $R$, resp.~of   the mod $p$-reduction $N_{\FF_p}$ of $N$. Recall that $H'\subset H''$ is a normal subgroup.  

Let $\chi'$ be the mod $p$ reduction of $\widetilde{\mu}'$. It is a cocharacter of $H':=G'_{\FF_p}$ and, hence, also of $H''$ defined over $\kappa$. By functoriality we get a map $$\gamma\colon H'-\hbox{{\rm Zip}}^{\chi'} \longrightarrow H''-\hbox{{\rm Zip}}^{\chi''}.$$The cocharacter $\chi'$ defines a pair of opposite parabolic subgroups $P'_{\pm}$ and $P''_{\pm}$ of $H'_\kappa$ and $H''_\kappa$ respectively. Let $\mathcal{Z}'=(P'_+,(P'_-)^{(p)},\varphi)$ and $\mathcal{Z}''=(P''_+,(P''_-)^{(p)},\varphi)$ be the two associated zip data. Recall that $H'-\hbox{{\rm Zip}}^{\chi'}\cong [E_{\mathcal{Z'}}\backslash H_k'] $ and  $H''-\hbox{{\rm Zip}}^{\chi''}\cong [E_{\mathcal{Z''}}\backslash H_k''] $. The inclusion $H'\subset H''$ identifies $P'_{\pm}=P''_{\pm} \cap H'_\kappa$ and induces a map of zip data $\mathcal{Z}'\to \mathcal{Z}''$ and a map of quotient stacks  $$\gamma\colon \bigl[E_{\mathcal{Z'}}\backslash H_k'\bigr]\longrightarrow \bigl[E_{\mathcal{Z''}}\backslash H_k''\bigr]$$ which provides a different realization of $\gamma$.

We define other zip data  for $H_k$ that will be relevant for our discussion. The first is $\mathcal{Z}:=\bigl(P,P_-^{(p)}, \varphi\bigr)$. The second is $\overline{\mathcal{Z}}:=\bigl(\overline{P}_\kappa,\overline{P}_{\kappa,-}^{(p)}, \varphi\bigr)$. Note that $E_{\overline{\mathcal{Z}}} \subset E_{\mathcal{Z}}$ so that we get a map $$\epsilon\colon  \bigl[E_{\overline{\mathcal{Z}}}\backslash H_k\bigr]\longrightarrow \bigl[E_{\mathcal{Z}}\backslash H_k\bigr].$$Identifying $H''_k$ with the standard Levi factor of $R_{\FF_p}$ (for our choices of the Borel $B$ and the torus $T$), we have an inclusion $H''_\kappa\subset R_\kappa \subset H_\kappa$ and $P''_+=\overline{P}_\kappa\cap H''_\kappa$ and $P''_-=\overline{P}_{-,\kappa}\cap H''_\kappa$. Then we get a map of stacks $$\upsilon \colon  \bigl[E_{\mathcal{Z}''}\backslash H_k''\bigr] \longrightarrow  \bigl[E_{\overline{\mathcal{Z}}}\backslash H_k\bigr]. $$

\begin{lemma}\label{lemma:gammazipsmooth} The morphisms $\gamma $, $\epsilon$, $\upsilon$ and, hence, the composite $\tau:=  \epsilon \circ \upsilon \circ \gamma$    are  smooth.

\end{lemma}
\begin{proof} 
We start with $\gamma$. Consider the push forward $F:=E_{\mathcal{Z}'}\backslash \bigl( E_{\mathcal{Z}''} \times_k H_k'\bigr)$ of the inclusion $E_{\mathcal{Z}'}\to E_{\mathcal{Z}''} $ via the group homomorphism $E_{\mathcal{Z}'}\to H_k'$. Here $E_{\mathcal{Z}'}$ acts on $ E_{\mathcal{Z}''} \times_k H_k'$ via $\bigl(\gamma,(\alpha,\beta)\bigr)\mapsto (\alpha \gamma^ {-1}, \gamma \beta)$.  The  projection $E_{\mathcal{Z}''} \times_k H_k' \to H_k'$ compsed with $H_k'\to [E_{\mathcal{Z}'}\backslash H_k'] $ defines a map $F\to [E_{\mathcal{Z}'}\backslash H_k']$. The inclusion $H_k'\subset H_k''$ induces a $E_{\mathcal{Z}''}$-equivariant map $\delta\colon F\to H_k''$ where $E_{\mathcal{Z}''}$ act on $F$ thoriugh left multiplication on the first factor of $E_{\mathcal{Z}''} \times_k H_k'$. We then have a commutative diagram 
$$\begin{matrix} 
 F & \stackrel{\delta}{\longrightarrow} & H_k'' \cr 
\downarrow & & \downarrow\cr
\bigl[E_{\mathcal{Z}'}\backslash H_k'\bigr] &\stackrel{\gamma}{\longrightarrow} & \bigl[E_{\mathcal{Z}''}\backslash H_k''\bigr] .\cr
\end{matrix}$$
For any scheme $S$, any $S$-valued point of the fibre product is, locally for the fppf topology, defined by the trivial $ E_{\mathcal{Z}'}$-torsor together with the map $E_{\mathcal{Z}'} \to H_k'$ given by right multiplication by $\beta\in H_k'(S)$, the trivial $ E_{\mathcal{Z}''}$-torsor together with the map $E_{\mathcal{Z}''} \to H_k''$ given by right multiplication by $\epsilon\in H_k''(S)$ and an elemnt $\alpha\in E_{\mathcal{Z}''}(T)$ such that $\epsilon=\alpha \beta\in H_k''(S)$.  Hence, $F$ is the fiber-product.

In order to prove that $\gamma$ is smooth, it suffices to show that the map $\delta$ is smooth. Recall that $H'$ is normal in $H''$ with reductive quotient $H'''$.  Hence $P'_{\pm}=P''_{\pm}\cap H'_\kappa$ is normal in $P''_{\pm}$ with quotient $H'''_\kappa$ as $\chi'$ defines the trivial character of $H'''$. Similarly $(P'_-)^{(p)}$ is normal in $(P''_-)^{(p)}$ with quotient $(H'''_\kappa)^{(p)}$. We deduce that $F$, via the right multiplication by $H_k'$, is an $H_k'$-torsor over $E_{\mathcal{Z'}}\backslash  E_{\mathcal{Z''}}$ which is identified with  $H'''_\kappa$ (embedded in $H'''_\kappa \times_k (H'''_\kappa)^{(p)}$ via the graph of Frobenius). Both $F$ and $H_k''$ are right $H_k'$-torosrs over $H_k'''$. In particular, also $F$ is a smooth $k$-scheme and the smoothness of $\delta$ is equivalent to the surjectivity of $\delta$ is on tangent spaces at closed points. The map $H_k'''\cong F/H_k'\to H_k''/H_k'\cong H_k'''$, induced by $\delta$, is Lang isogeny $h\mapsto h (h^{p})^{-1}$ which is sujective and separable. Hence $\delta$ induces an  isomorphism on tangent spaces as wanted..

As the morphism $\epsilon$ is a quotient morphism of smooth stacks, it is smooth. 

We are left to prove the smoothness of $\upsilon$. Let $M:=E_{\mathcal{Z}''}\backslash \bigl( E_{\overline{\mathcal{Z}}} \times_k H_k''\bigr)$ be the push forward of  the inclusion $E_{\mathcal{Z''}}\to E_{\overline{\mathcal{Z}}} $ via the group homomorphism $E_{\mathcal{Z''}} \to H''_\kappa$. The inclusion  $H''\subset H$ defines a $E_{\overline{\mathcal{Z}}}$-equivariant map $\tau\colon M\to H_k$ and arguing as before we get that $M$  is the fiber product
$$\begin{matrix} 
 M & \stackrel{\tau}{\longrightarrow} & H_k \cr 
\downarrow & & \downarrow\cr
\bigl[E_{\mathcal{Z}''}\backslash H_k''\bigr] & \stackrel{\upsilon}{\longrightarrow} & \bigl[E_{\overline{\mathcal{Z}}}\backslash H_k\bigr] .\cr
\end{matrix}$$
The smoothness of $\upsilon$ is then equivalent to the smoothness of $\tau$. As $\tau$ is a morphism of schemes smooth  over $k$, it suffices to show that it is surjective on tangent spaces at $k$-valued points and, by right translatiion by $E_{\overline{\mathcal{Z}}} \times_k H_k''$,  it suffices to show this at the $0$-sections.  Recall that the parabolic subgroup $R$ of $G_{\ZZ_p}$ is defined over $\ZZ_p$ and it has unipotent radical $U_{\ZZ_p}$ also defined over $\ZZ_p$. Let $U_{\FF_p}$ be its reduction modulo $p$. Then $R_{\FF_p} \cap R_{\FF_p,-}=H''$ and it follows that ${\rm Lie} H={\rm Lie} H'' \oplus {\rm Lie} U_{\FF_p} \oplus {\rm Lie} U_{\FF_p,-}$. Notice that $U_{\FF_p}$ is contained in the unipotent radical of $\overline{P}_\kappa$. Similarly $U_{\FF_p,-}$ is the unipotent radical of  $\bigl(R_{\FF_p,-} \bigr)^{(p)}$ and is contained in the unipotent radical of $\overline{P}_{\kappa,-}^{(p)}$. Hence,  ${\rm Lie} U_{\FF_p} \oplus {\rm Lie} U_{\FF_p,-}$ is contained in the image of ${\rm Lie}  E_{\overline{\mathcal{Z}}} $. Thus $\tau$ induces a surjective map on tangent spaces at the $0$-sections as wanted.

\end{proof}

\begin{remark}  In the proof of Lemma \ref{lemma:gammazipsmooth} the fact that the weight cocharacter is defined over $\QQ$, so that  $R_{\ZZ_p}$ and its unipotent radical are defined over $\ZZ_p$, is the key point of the proof. 

\end{remark}

Over $S'$ one has an abelian scheme $B$. The covariant  Dieudonn\'e module of its mod $p$ reduction $B_0$ defines an associated map $$\zeta'\colon S'_0\to  H'-\hbox{{\rm Zip}}^{\chi'}$$which is smooth by \cite[Cor. 3.12]{Chao}.  As $\widetilde{S}_0^\Sigma\to S'_0 $ is smooth, the composite map $\zeta^{',\Sigma}\colon \widetilde{S}_0^\Sigma\to S'_0 \to H'-\hbox{{\rm Zip}}^{\chi'}$ is smooth as well.

\begin{proposition}\label{prop:zetSigmaA0} The map  $\zeta^\Sigma\colon {\rm Spec}(A_0)\to  H-\hbox{{\rm Zip}}^\chi$ of Lemma \ref{lemma:zetaSigma}  is the restriction to ${\rm Spec}(A_0)$ of the  morphism  $\tau \circ \zeta'$ with $\tau:=\gamma \circ \epsilon \circ \upsilon$. In particular, it  is surjective on tangent spaces at the closed point. 

\end{proposition}
\begin{proof} The last statement follows from the fact that $\gamma \circ \epsilon \circ \upsilon \circ \zeta'$ is smooth as it is the composite of smooth morphisms by Lemma \ref{lemma:gammazipsmooth}.

As remarked in \S\ref{sec:boundary} the weight filtration on  $\cH_A$ arises from a filtration by Frobenius crystals, that is, Frobenius preserves the weight filtration. We conclude from  Proposition \ref{prop:reductionoverP} that the element $g_{t,0}$  of Lemma \ref{lemma:zetaSigma} lies in $R_{\FF_p}(A_0)$. Let $g_{t,0}''$ be the image of $g_{t,0}$ in the quotient $H''$ of $R_{\FF_p}$. The image of the class of $g_{t,0}''$ via $\nu$ coincides with the class of $g_{t,0}$ as the kernel $U_{\FF_p}$ of the quotient map $R_{\FF_p}\to H''$ is contained in the unipotent radical of $\overline{P}_{\kappa}$ and hence in $E_{\overline{\mathcal{Z}}}$ (as explained in the proof of of Lemma \ref{lemma:gammazipsmooth}).  Let $$\zeta''\colon  {\rm Spec}(A_0)\longrightarrow \bigl[E_{\mathcal{Z}''}\backslash H_k''\bigr] $$be the map defined by $g_{t,0}''\in H''(A_0) $. Then,  the map $\zeta^\Sigma\colon {\rm Spec}(A_0)\to  H-\hbox{{\rm Zip}}^\chi$  factors as $$\zeta^\Sigma= \epsilon \circ \nu \circ \zeta''.$$     Possibly extending the residue field of $A$ we may assume that the torus $T$ and the lattice $Y$ of the degenerating $1$-motive $M_U^\vee$ over $A$ are trivial (see \S\ref{sec:boundary} for the notation). Taking generators of ${\rm Lie }(T)$ coming from a basis of cocharacters of $T$ and a $\ZZ$-basis of $Y$, that defines a basis of the universal extension of $Y$, Frobenius on  $W_0 \cH_A$ becomes multiplication by $p$ and Frobenius on ${\rm Gr}_{2} \cH_A$ is the identity matrix. We conclude that, after a finite \'etale extension of $A$, we can assume that the projection of  $g_{t,0}''$ onto ${\rm Aut}(W_0 \cH_{A_0}) \times {\rm Aut}({\rm Gr}_{2} \cH_{A_0})$ is the identity so that it lies in ${\rm Aut}({\rm Gr}_{2} \cH_{A_0})$, the automorphism group of the Dieudonn\'e module of $B_0$ base changed to ${\rm Spec}(A_0)$. Hence, the morphism  $\zeta''$ factors via $\zeta'$ restricted to ${\rm Spec}(A_0)$ composed with $\gamma$ as wanted. 

\end{proof}

\subsection{Ekedahl-Oort strata at the boundary}

It follows from \cite[Thm. 5.3.1]{MP} that there exists a minimal compactification $S_0^\ast$ and a map $\pi_0\colon S_0^\Sigma \to S_0^\ast$ such that the structure at the boundary is induced by the fibration described in \S\ref{sec:boundary} (here the index $0$ in $S_0^\Sigma$ denotes as usual the mod $p$ special fiber). More precisely, the boundary components of $S_0^\ast$ are a disjoint union of the special fiber of integral models of lower dimensional smooth Shimura varieties of the type $S_0'$, as considered in \S\ref{sec:boundary}, indexed by so called cusp label representatives. For any such component the map $\pi^{-1}(S_0')\to S_0'$   is isomorphic to the map $\partial \widetilde{S}^\Sigma_0 \to S_0$ defined by restricting the morphism  $\widetilde{S}^\Sigma_0 \to S_0$ of \S\ref{sec:boundary} to a closed subscheme $\partial \widetilde{S}^\Sigma_0\subset \widetilde{S}^\Sigma_0$. Then, we have the following description of the Ekedahl-Oort stratification on the boundary of $S_0^\Sigma$:

\begin{corollary}\label{cor:pointnotintheimageofzetaSigma} For every boundary component $S_0'$ of $S_0^\ast$ as above, the map $\zeta^\Sigma$ restricted to $\pi^{-1}(S_0')$ is the composite of 

\begin{itemize}

\item[i.] the projection $\pi^{-1}(S_0') \to S_0'$;

\item[ii.] the map $\zeta'\colon S_0'\to H'-\hbox{{\rm Zip}}^{\chi'}$ defined in \cite{Chao};

\item[iii.] the map  $\tau\colon H'-\hbox{{\rm Zip}}^{\chi'} \to H-\hbox{{\rm Zip}}^{\chi}$ of Lemma \ref{lemma:gammazipsmooth}.

\end{itemize}

In particular, $\zeta^\Sigma$ restricted to $\pi^{-1}(S_0')$ is constant on the geometric fibers of $\pi$ and the induced map on underlying topological spaces $$\vert \zeta^\Sigma \vert\colon \vert \pi^{-1}(S_0') \vert \to \vert  H-\hbox{{\rm Zip}}^{\chi} \vert$$ has open image and does not contain the closed point of $\vert H-\hbox{{\rm Zip}}^{\chi}\vert$.

\end{corollary}
\begin{proof} Statements (i)--(iii) follow from Proposition \ref{prop:zetSigmaA0}. The map $\tau$ is smooth and, hence, its image on the underlying topological spaces is open, i.e., it is closed under generalization. If it contained the closed point, it would be surjective and any maximal chain of points of $\vert H-\hbox{{\rm Zip}}^{\chi} \vert$ would be the image of a maximal chain of points of $\vert H'-\hbox{{\rm Zip}}^{\chi'} \vert$. This can not be as any maximal chain of points in $\vert H'-\hbox{{\rm Zip}}^{\chi'} \vert$ has length equal to the dimension of $S_0'$ which is strictly smaller than the dimension of $S_0$ which is equal to the length of a maximal chain of points  of $\vert H-\hbox{{\rm Zip}}^{\chi} \vert$. This implies the last statement. 

\end{proof}

\subsection{A result on zips in mixed chracteristics}

The cocharacter $\widetilde{\mu}$ of $G_{\WW(\kappa)}$, whose reduction is $\chi$, defines two parabolic subgroups $\widetilde{P}=\widetilde{P}_+$ and $\widetilde{P}_-$ of $G_{\WW(\kappa)}$. The reduction of $\widetilde{P}_\pm$ is $P_\pm$. Let $\widetilde{T}\subset \widetilde{B}\subset G_{\ZZ_p}$ be a Borel subgroup and a maximal torus, defined over $\ZZ_p$. By replacing $\widetilde{\mu}$ by a conjugate cocharacter, that might be defined over an extension $\WW(\kappa)\subset \WW(\kappa')$,  we may and will assume that $\widetilde{P}_+$ and  $\widetilde{P}_-$ contain $\widetilde{B}$. There is a common Levi subgroup $\widetilde{L}=\widetilde{P}_+ \cap \widetilde{P}_-$.  Finally we let $\widetilde{Q} \subset G_{\WW(\kappa)}$ be the parabolic subgroup $\widetilde{P}_-^{(p)}$ defined by base change of $\widetilde{P}_-\subset G_{\WW(\kappa)}$ via  Frobenius using that  $G_{\WW(\kappa)}^{(p)}= G_{\WW(\kappa)}$ as $G_{\WW(\kappa)}$ is defined ove $\ZZ_p$. 

\

Suppose we are given a  $\WW(k)$-algebra $D$  endowed with a lift of Frobenius $\varphi$. Let $\Phi\colon G(D) \to G(D)$  be the Frobenius map $g\mapsto g^{(p)}$; using the embedding $G_{\ZZ_p}\subset {\rm GSp}_{\Lambda}\cong {\rm GSp}_{2g}$ we view $g\in G(D)$ as a $(2g)\times (2g)$-matrix with coefficients in $D$ and we let $\Phi(g)=g^{(p)}$ be the matrix where we apply $\varphi$ to the coefficients. It is still an element of $G(D)$ since the equations defining the closed subgroupscheme  $G_{\ZZ_p}\subset {\rm GSp}_{2g,\ZZ_p}$ have coefficients in $\ZZ_p$. 
Define $\widetilde{E}_{\mathcal{Z}_{\rm EO}}(D)$  to be the subgroup of pairs $(g,h)\in G(D)\times G(D)$ with  $g\in \widetilde{P}(D)$, $h\in \widetilde{Q}(D)$ and such that the map $$\overline{\Phi}\colon \bigl(\widetilde{P}(D)/{\rm R}_u \widetilde{P}(D) \bigr)= \widetilde{L}(D) \longrightarrow \bigl(\widetilde{L}\bigr)^{(p)}(D)=\bigl(\widetilde{Q}(D)/ \bigl({\rm R}_u \widetilde{Q}(D)\bigr)\bigr)$$sends the class of $g$ to the class of $h$. It acts on the left on  $G(D)$ via $(a,b) \cdot g= a g b^{-1}$.

\

Recall the notation of \S\ref{sec:boundary}; $A$ is the compelte local ring of $S^\Sigma$ at a closed point $x_0$. Frobenius on $\cH_A$ defines an element $g_t\in G(A)$.  Assume that $\zeta^\Sigma(x_0)$ is the class of $\widetilde{w} \widetilde{z}^{-1}$ for $w\in{^I}W$. See  Lemma \ref{cor:HZipIW}. For every $\WW(k)$-point $x$ of $A$ we let $\cH_x$ be the pull-back of $\cH_A$. It is a $\WW(k)$-module, endowed with a weight filtration, Hodge filtration and a Frobenius linear map $F$. Proceeding as in the discussion before Lemma \ref{lemma:zetaSigma} Frobenius and a splitting of the Hodge filtration produce an element $g_x\in  G\bigl(\WW(k)\bigr)$. Then,

\begin{proposition}\label{prop:LieE} There exists a $\WW(k)$-point $x$ of $A$, that reduces to $x_0$ modulo $p$ and  generically lies in  $S\subset S^\Sigma$, such that the class of $g_x$  in $\widetilde{E}_{\mathcal{Z}_{\rm EO}}(\WW(k)) \backslash G\bigl(\WW(k)\bigr)$ is $\widetilde{w} \widetilde{z}^{-1}$.

\end{proposition}

\begin{proof}   First of all we claim that it suffices to construct $x$ wthout assuming that  it generically lies in  $S$.
For any $\WW(k)$-valued point $y$ of $A$, Proposition \ref{prop:reductionoverP} and the fact that Frobenius on $\cH_y$ preserves the weight filtration imply that $g_y\in R_{\ZZ_p}\bigl(\WW(k)\bigr)$. The unipotent radical $U_{\ZZ_p}$ of $R_{\ZZ_p}$ is contained in the unipotent radical of $\widetilde{P}$ so that the class of $g_y$ dpends only on the class $\overline{g}_y$ in the reductive quotient $G''_{\ZZ_p}\bigl(\WW(k)\bigr)=R_{\ZZ_p}\bigl(\WW(k)\bigr)/U_ {\ZZ_p}\bigl(\WW(k)\bigr)$.
Since Frobenius  on  $W_0 \cH_y$ with respect to a basis coming form a basis of cocharacters of the torus $y^\ast(T)$ can be taken to be multiplication by $p$  and Frobenius on ${\rm Gr}_{2} \cH_y$ with respect to a basis coming form a basis of $y^\ast(Y)$  is the identity matrix, in fact $\overline{g}_y$ can be taken in $G'_{\ZZ_p}\bigl(\WW(k)\bigr)$ and this dpends only on Frobenius defined on the factor ${\rm Gr}_{1} \cH_y$ which is the Frobenius crystal over $\WW(k)$ defined by pulling back the abelian scheme $B$ via $y$. In particular,  if we let $y'\in S'(\WW(k))$ be the image of $y$   in the fibration of \S \ref{sec:boundary} and if we take any $\WW(k)$-valued point $v$ of $A$ in the fiber of $y'$ then the class of $g_y$ and of $g_v$ in $\widetilde{E}_{\mathcal{Z}_{\rm EO}}(\WW(k)) \backslash G\bigl(\WW(k)\bigr)$ are the same.   As we have $v$'s in the fiber over $y'$ that generically lie in $S$, the claim is proven.

\smallskip

We are left to prove the Proposition dropping the requirement that $x$ lies generically in  $S$. We construct inductively $\WW(k)$-points $(y_n)_{n\in \mathbb{N}}$ such that $y_0\equiv x_0$ modulo $p$ and  $y_{n+1}\equiv y_n$ modulo $p^{n+1}$ for every $n$  and $g_{y_n} \equiv \widetilde{w} \widetilde{z}^{-1} $  in $\widetilde{E}_{\mathcal{Z}_{\rm EO}}(\WW_{n+1}(k)) \backslash G\bigl(\WW_{n+1}(k)\bigr)$.  As the maps $\alpha_n\colon \widetilde{E}_{\mathcal{Z}_{\rm EO}}(\WW_{n+1}(k)) \to E_{\mathcal{Z}_{\rm EO}}(\WW_n(k))$ and $\beta_n \colon G(\WW_{n+1}(k))\to G(\WW_n(k))$ are surjective, the limit point $x=\displaystyle{\lim_{n\to \infty}} y_n$ satsfies the requirements. 

Take $y_0$ to be any lift of $x_0$. Assume that $y_n$ is constructed. Then, possibly after multiplying $g_n:=g_{y_n}$ by an element of $\widetilde{E}_{\mathcal{Z}_{\rm EO}}(\WW_{n+1}(k)) $ we can asssume that $g_n  \widetilde{z} \widetilde{w}^{-1}$ lies in the kernel of $\beta_n$ which is isomorphic to ${\rm Lie} H_k$ (recall that $H$ is the reduction of $G_{\ZZ_p}$ modulo $p$). Thus such class defines an obstruction element ${\rm Obs}(y_n)\in {\rm Lie} H_k$

Let  ${\rm Lie} (E_{\mathcal{Z}_{\rm EO}}) $ be the subset of pairs $(a,b)\in  {\rm Lie} (P)\otimes k \oplus {\rm Lie}( Q)\otimes k$ such that  the map $1 \otimes \varphi \colon {\rm Lie} L \otimes k \longrightarrow {\rm Lie} \bigl(L\bigr)^{(p)}\otimes k $ sends the class of $a$  modulo ${\rm Lie} {\rm R}_u P \otimes k$ to  the class of $b$ modulo ${\rm Lie} {\rm R}_u Q \otimes k$. The krnel of $\alpha_n$ is isomorphic to ${\rm Lie} (E_{\mathcal{Z}_{\rm EO}}) $.
Write   ${\rm Lie} (P)=  {\rm Lie} (L) \oplus    {\rm Lie} ({\rm R}_u P)$ and ${\rm Lie} (Q)=  {\rm Lie} (L^{(p)}) \oplus    {\rm Lie} ({\rm R}_uQ)$. For $a\in{\rm Lie} (P)$ write $a=a_0 + a_1$ for the corresponding decomposition. Similarly for  $b\in{\rm Lie} (Q)$ write $b=b_0 + b_1$. Then $(a,b)\in {\rm Lie} (E_{\mathcal{Z}_{\rm EO}}(\WW(k))$ if and only if $\varphi(a_0)=b_0$. As Frobenius on Lie algebars is the zero map, this amounts to require that $b_0=0$,  that is  ${\rm Lie} (E_{\mathcal{Z}_{\rm EO}}(k)$  consist of pairs $(a,b)$ such that $b_0=0$. Consider the map of Lie algebars over $k$: $$\nu\colon {\rm Lie} (E_{\mathcal{Z}_{\rm EO}}) = {\rm Lie} (L) \oplus ( {\rm R}_u P)  \oplus \Lie ( {\rm R}_u Q)\longrightarrow {\rm Lie} G_k, \qquad (m,n,z)\mapsto m+n-z.$$The element  ${\rm Obs}(y_n)$ lies in the image of $\nu$ if and only if after multiplying $g_n:=g_{y_n}$ by an element of $\widetilde{E}_{\mathcal{Z}_{\rm EO}}(\WW_{n+1}(k)) $ we have $g_{y_n} \equiv \widetilde{w} \widetilde{z}^{-1} $  in $\widetilde{E}_{\mathcal{Z}_{\rm EO}}(\WW_{n+1}(k)) \backslash G\bigl(\WW_{n+1}(k)\bigr)$. In this case we can take $y_{n+1}=y_n$.

Recall that by construction we have a lift  ${\rm Spec}(A_0)\to H_k$ of the map $\zeta^\Sigma\colon S_0^\Sigma \to G-\hbox{{\rm Zip}}^{\chi}$. It induces a map from the tangent space $T_{x_0}$ of $S^\Sigma_0$ at $x_0$ to the tangent space $T_{H, x_0}$  of $H_k$ at the image $c_0$ of the closed point of $A_0$, that we identify with ${\rm Lie} H_k$ by translating by $c_0^{-1}$.  The smoothness of $\zeta^\Sigma$ in Theorem \ref{theorem:zetasigmasmooth} implies that the induced map  $$\theta\colon {\rm Lie} (E_{\mathcal{Z}_{\rm EO}}) \times T_{x_0} \to {\rm Lie} H_k $$is surjective.  The possible lifts of $y_n$ to a $\WW_{n+1}(k)$-valued point of $A$ is a principal homogeneous space under $T_{x_0}$ as we assume that our toroidal compactification $S^\Sigma$ is smooth. Let $y_n'$ be a lift corresponding to an element $s\in T_{x_0}$. Then 
${\rm Obs}(y_n')={\rm Obs}(y_n) + \theta((0,s))$. By the surjectivity of $\theta$ we can then find an $s$, and hence a corresponding lift $y_{n+1}=y_n'$, and an element $b\in {\rm Lie} (E_{\mathcal{Z}_{\rm EO}})$ such that ${\rm Obs}(y_n')=\theta((b,0)$. This implies that after multiplying $g_{n+1}:=g_{y_{n+1}}$ by an element in the kernel of $\alpha_n$ we have that $g_{y_{n+1}} \equiv \widetilde{w} \widetilde{z}^{-1} $  in $\widetilde{E}_{\mathcal{Z}_{\rm EO}}(\WW_{n+1}(k)) \backslash G\bigl(\WW_{n+1}(k)\bigr)$. This concludes the proof of the inductive step and of the Proposition.

\end{proof}

\section{On the mod $p$ Hodge-Tate period map}

We describe the map $q_2\colon \cS(p^\infty)\to \mathrm{F}\ell_{G,\mu^{-1},0}$ using certain formal models introduced in \cite{PS}.

\smallskip

Thanks to \cite[Thm. IV.6.7]{FC} there exists a smoooth toroidal compactifiction  $S_{K_g,\QQ_p}^{\Sigma_g}(p^n)$ of the full level $p^n$ Siegel modular scheme $S_{K_g,\QQ_p}(p^n)\to S_{K_g,\QQ_p}$, associated to the finite admissible cone decomposition  $\Sigma_g$ fixed in (\ref{ass}). It defines a ${\rm Sp}_{2g}(\ZZ/p^n\ZZ)$-cover  $$S_{K_g,\QQ_p}^{\Sigma_g}(p^n)\to S_{K_g,\QQ_p}^{\Sigma_g}.$$Let $T$ be the $p$-adic Tate module fo the universal abelian scheme over $S_{\QQ_p}$. One has  Tate tensors $(s_{\rm et,\alpha})_\alpha$ in the tensor algebra of $T$. Over the fiber product $S_{\QQ_p}\times_{S_{K_g,\QQ_p}^{\Sigma_g}} S_{K_g,\QQ_p}^{\Sigma_g}(p^n)$ the \'etale sheaf  $T/p^n T\cong \Lambda/p^n \Lambda$ is trivialized and it extends to a  constant sheaf $T_n\cong \Lambda/p^n\Lambda$ over $S_{\QQ_p}^{\Sigma}\times_{S_{K_g,\QQ_p}^{\Sigma_g}} S_{K_g,\QQ_p}^{\Sigma_g}(p^n)$. Let $S_{\QQ_p}^{\Sigma}(p^n)$ be  the closed subscheme $$S_{\QQ_p}^{\Sigma}(p^n)\subset S_{\QQ_p}^{\Sigma}\times_{S_{K_g,\QQ_p}^{\Sigma_g}} S_{K_g,\QQ_p}^{\Sigma_g}(p^n)$$such that the isomorphism $t_n\colon T_n\cong \Lambda/p^n\Lambda$ sends  the mod $p^n$-reduction of the tensors $(s_{\rm et,\alpha})_\alpha$  in the tensor algebra of $T_n$ to the tensors $(s_\alpha)_\alpha$ of $\Lambda/p^n\Lambda$. 

Let $S_{K_g}^{\Sigma_g}(p^n)$ be the normalization of the toroidal compactification $S_{K_g}^{\Sigma_g}$ of the Siegel modular scheme  in $S_{K_g,\QQ_p}^{\Sigma_g}(p^n)$. Similarly let  $S(p^n)$, resp.~$S^\Sigma(p^n)$ be the normalization of $S$, resp.~$S^\Sigma$ in   $S_{\QQ_p}^{\Sigma}(p^n)$. It is finite over $S$,  resp.~$S^\Sigma$ by construction. By \cite[Prop. 1.5]{PS} we have a map  $$d\log_n\colon T_n^\vee(1) \to \Omega \otimes \bigl(\cO_{S^\Sigma(p^n)}/ p^n \cO_ {S^\Sigma(p^n)}\bigr)$$given as follows. First of all the Cartier dual of the $p^n$-torsion  of the universal semiabelian scheme $G$ over $S_{\QQ_p}^{\Sigma_g}(p^n)$ is a quotient of $T_n^\vee(1)$ through which $d\log_n$ factors. Any such element extends uniquely to an element $\gamma\colon G[p^n] \to \GG_{m,S(p^n)^{\Sigma}}$, defined over the fppf topology of $S^{\Sigma}(p^n)$. The pull-back of the standard invariant differential on $\GG_{m}$ defines an element of the invariant differentials of $G[p^n]$ that coincide with the mod $p^n$-reduction of the invariant differentials of $G$.   Finally we use \cite[Thm. 1.1]{FC} to identify the sheaf of invariant differentials $\Omega_G$ of the universal semiabelian scheme $G$ over $S^\Sigma$, used in \cite{PS}, with the first piece of the  Hodge filtration $\Omega\subset {\rm H}_{\rm dR}^1$ of the canonical extension ${\rm H}_{\rm dR}^1$ described in \S\ref{sec:EO} and used in \cite{GK}. We can now state \cite[Prop. 1.13]{PS} and \cite[\S 2.3.1]{PS}.

Denote by $S^\Sigma_0(p^n)$ and  $S_{K_g,0}^{\Sigma_g}(p^n)$ the mod $p$-reductions of $S^\Sigma(p^n)$ and $S_{K_g}^{\Sigma_g}(p^n)$ respectively.  Set $\cS^\Sigma(p^n)$ and $\cS_{K_g}^{\Sigma_g}(p^n)$ to be the  formal schemes defined by completing  $S^\Sigma(p^n)$ along $S^\Sigma_0(p^n)$, resp.~of $S_{K_g}^{\Sigma_g}(p^n)$ along $S_{K_g,0}^{\Sigma_g}(p^n)$. Set $\tau\colon  \cS^\Sigma(p^n)\to \cS_{K_g}^{\Sigma_g}(p^n)$ be the morphism induced by the map $S^\Sigma(p^n) \to  S_{K_g}^{\Sigma_g}(p^n)$. Let $\zeta_p$  a primitive $p$-th root of unity.

\begin{proposition}\label{prop:HT1} {{\rm [\cite{PS}]}}\enspace Given integers $n\geq \frac{g}{p-1}+1$ and $1\leq m\leq n-1$ there exist:

\begin{itemize}

\item[i.] admissible normal formal schemes $\overline{\cS}^\Sigma(p^n)$ and $ \overline{\cS}_{K_g}^{\Sigma_g}(p^n)$, projective morphisms $\overline{\cS}^\Sigma(p^n)\to \cS^\Sigma(p^n)$ and $ \overline{\cS}_{K_g}^{\Sigma_g}(p^n) \to  \cS_{K_g}^{\Sigma_g}(p^n)$ inducing an isomorphism on the adic generic fiber and a morphism $\overline{\tau}\colon \overline{\cS}^\Sigma(p^n) \to \overline{\cS}_{K_g}^{\Sigma_g}(p^n)$ making the following diagram commute:

$$\begin{matrix} \overline{\cS}^\Sigma(p^n)& \stackrel{\overline{\tau}}{ \longrightarrow} &  \overline{\cS}_{K_g}^{\Sigma_g}(p^n)\cr
\big\downarrow & & \big\downarrow \cr
\cS^\Sigma(p^n) & \stackrel{\tau}{\longrightarrow} & \cS_{K_g}^{\Sigma_g}(p^n); \cr
\end{matrix}$$

\item[ii.]  a $\cO_{\overline{\cS}^\Sigma(p^n)}$-submodule $\Omega^{\rm mod}\subset \Omega$,  locally free  of rank $g$, such $(\zeta_p-1) \Omega \subset  \Omega^{\rm mod} $. There is a similar statement over  $\overline{\cS}_{K_g}^{\Sigma_g}(p^n)$ and the sheaves over $\overline{\cS}^\Sigma(p^n)$ are obtained from those on $\overline{\cS}_{K_g}^{\Sigma_g}(p^n)$ by pull-back via the morphism $ \overline{\cS}^\Sigma(p^n) \to  \overline{\cS}_{K_g}^{\Sigma_g}(p^n)$ in (i);

\item[iii.] a surjective map  $${\rm HT}_{m}\colon \bigl(\Lambda^\vee/p^m \Lambda^\vee\bigr)\otimes \cO_{\overline{\cS}^\Sigma(p^n)}(1)  \to \Omega^{\rm mod}/p^m \Omega^{\rm mod}$$ induced by the map $d\log_n$ and the identification $t_n\colon T_n\cong \Lambda/p^n\Lambda$. There is a similar statement over  $\overline{\cS}_{K_g}^{\Sigma_g}(p^n)$ and the map  ${\rm HT}_{m}$ over $\overline{\cS}^\Sigma(p^n)$ is obtained from the one on $\overline{\cS}_{K_g}^{\Sigma_g}(p^n)$ by base change via the morphism $ \overline{\cS}^\Sigma(p^n) \to  \overline{\cS}_{K_g}^{\Sigma_g}(p^n)$ in (i).

\end{itemize}

The sheaf $\Omega^{\rm mod}$ is independent of $n$ and $m$ and for $1\leq m' \leq m \leq n-1$ the map ${\rm HT}_{m'}$ is the reduction of ${\rm HT}_m$ modulo $p^{m'}$. 
\end{proposition}

Let $\overline{S}_0^\Sigma(p^n)$ be the scheme over $\kappa$  given by the mod $p$ special fiber of $\overline{\cS}^\Sigma(p^n)$. Denote by $\bigl(\overline{S}_0^\Sigma(p^n)\bigr)^{\rm red} $ the reduced scheme associated to $S_0^\Sigma(p^n)$

\

Recall that, using the Hodge cocharacter, we have defined a cocharacter $\chi$ of $(G_{\ZZ_p})_\kappa=H_\kappa$ and hence a decomposition $\Lambda \otimes \kappa=\Lambda_0\otimes \kappa \oplus \Lambda_{-1} \otimes \kappa$.  Taking the sequence $\Lambda_{-1} \otimes \kappa \subset \Lambda \otimes \kappa$ we get a base point in  $\mathrm{F}\ell_{G,\mu^{-1},0}$, i.e., an identification as schemes over $\kappa$: $$\mathrm{F}\ell_{G,\mu^{-1},0}\cong \mathrm{F}\ell_{H,\chi^{-1}} \cong P_-\backslash H_\kappa.$$

\begin{corollary}\label{cor:piHT0} The subsheaf ${\rm Ker}({\rm HT}_1)$ provided by Proposition \ref{prop:HT1} defines a map of schemes over $\kappa$: $$\pi_{\rm HT,0}\colon\bigl(\overline{S}_0^\Sigma(p^n)\bigr)^{\rm red}  \to P_-\backslash H_\kappa.$$

\end{corollary}
\begin{proof} Let $P_g^\pm$ be the parabolic of ${\rm GSp}_{2g,\kappa}$ induced by the cocharacter $\chi^\pm$  via the inclusion $G_{\kappa}\subset {\rm GSp}_{2g,\kappa}$. Then ${\rm Ker}({\rm HT}_1)$ provides a map $\overline{S}_0^\Sigma(p^n) \to P_g^-\backslash {\rm GSp}_{2g,\kappa}$. The Hodge-Tate period map $\pi_{\rm HT}$ is functorial with respect to morphims of Shimura data $(G,X)\to \bigl({\rm GSp}_{2g},\mathbb{H}_g\bigr) $ by \cite[Thm. 2.1.2 \& Thm.~2.1.3]{CS}. By  \cite{PS}  the integral Hodge-Tate period map for the Siegel variety $\overline{\cS}^{\Sigma_g}_{K_g}(p^\infty)\to \mathrm{F}\ell_{{\rm GSp}_{2g},\mu^{-1}}$ is compatible with the one constructed by Scholze. Hence the image of $\overline{\cS}^\Sigma(p^\infty)$ is contained in $\mathrm{F}\ell_{G,\mu^{-1}}$. 
By the compatibilities of ${\rm HT}_1 $ over $\overline{\cS}^\Sigma(p^n)$ and $\overline{\cS}^{\Sigma_g}_{K_g}(p^n)$    respectively (see  Proposition \ref{prop:HT1} (iii)) the composite morphism on special fibers $$\overline{S}_0^\Sigma(p^n)\to \overline{S}^{\Sigma_g}_{K_g,0}(p^n) \to \mathrm{F}\ell_{{\rm GSp}_{2g},\mu^{-1},0},$$provided by  ${\rm Ker}({\rm HT}_1)$, factors through the image of the composite
$\mathrm{F}\ell_{G,\mu^{-1},0} \to \mathrm{F}\ell_{{\rm GSp}_{2g},\mu^{-1},0}\cong P_g^-\backslash {\rm GSp}_{2g,\kappa}$. Such image coincides with   $P_-\backslash H_\kappa$ as  $P_-=H_\kappa\cap P_g^-$ and  the map $P_-\backslash H_\kappa \to P_g^-\backslash {\rm GSp}_{2g,\kappa}$ is a closed immersion. The claim follows.

\end{proof}

We need to restate the Corollary using \S\ref{sec:variant} in order to express the fine Deligne-Lusztig stratification on $\mathrm{F}\ell_{H,\chi^{-1},0}$ in terms of the Weyl group $W$.  Let $I_-$ be the type opposite to the type $I$ of $P$. Thanks to Lemma \ref{lemma:EZEZ'} we have an isomorphism as schemes over $k$ $$P_- \backslash H_k \cong  P_{I_-} \backslash H_k, \quad h\mapsto  \widetilde{y}^{-1} h  \widetilde{z}.$$Via the projection map $$P_{I_-} \backslash H_k \to \bigl[ E_{\mathcal{Z}_-}\backslash H_k\bigr]$$we get a stratification on $  P_-\backslash H_k$ indexed by the elements of ${^{I_-}}W$ with closure relation described by the order relation  $\preccurlyeq$; given two elements $w$ and $w'\in  {^{I_-}}W$ we write $w \preccurlyeq w' $ if and only if there exists $x\in W_{I_-}$ such that $x w \varphi(x)^{-1} \leq w'$ (for the  Bruhat order on $W$). We let 

\begin{equation}\label{eq:gammabar} \overline{\gamma}\colon \mathrm{F}\ell_{H_k,\chi^{-1},0}\cong P_- \backslash H_k \cong  P_{I_-} \backslash H_k\longrightarrow  \bigl[ E_{\mathcal{Z}_-}\backslash H_k\bigr]\end{equation}
be the composite map.

For later purposes we prove the following result about the modulo $p$ Hodge-Tate period map.

\begin{proposition}\label{prop:HTonto} The map $\pi_{\rm HT,0}$ is surjective. 

\end{proposition}
\begin{proof} 
Let $\overline{S}^\Sigma(p^n) \to S^\Sigma(p^n)$ be the projective scheme over $\WW(\kappa)$ and the morphism to $S^\Sigma(p^n)$ algebrizing the morphism of  formal scheme $\overline{\cS}^\Sigma(p^n) \to \cS^\Sigma(p^n)$ of Proposition \ref{prop:HT1}(i). Such morphism is projective and it an isomorphism over $\QQ_p$.  Consider the invertible sheaf $L$ on $\overline{S}^\Sigma(p^n)$ algebrizing the sheaf $\det \Omega^{\rm mod}$ on $\overline{\cS}^\Sigma(p^n)$  and let $L_0$ be the induced invertible sheaf on $\overline{S}_0^\Sigma(p^n)$. The sheaf $L$ is base point free and after inverting $p$ it coincides with the Hodge bundle by Proposition \ref{prop:HT1}(ii). In particular, its restriction to the  $S(p^n)_{\QQ_p}$, identified with an open subscheme of $\overline{S}^\Sigma(p^n) $, is ample. Take a large enough power $L^h$ which is  very ample on $S(p^n)_{\QQ_p}$ . Choosing a $\WW(\kappa)$-basis of global sections ww get a map $$\Psi\colon \overline{S}^\Sigma(p^n) \longrightarrow \mathbb{P}^N_{\WW(\kappa)}.$$As $\overline{S}^\Sigma(p^n)$ is projective over $\WW(\kappa)$, the morphism $\Psi$ is projective and its image $D^\Sigma(p^n)$  is closed.   After inverting $p$, it is an immersion on  $S(p^n)_{\QQ_p}$ so that the image of $\Psi$, after inverting $p$, has dimension equal to the dimension $d$ of $S_{\QQ_p}$.  In particular, $D_0^\Sigma(p^n):=D^\Sigma(p^n) \cap \mathbb{P}^N_{\kappa}$ is equidimensional of dimension $d$. Denote by $\Psi_0$ the morphism induced by $\Psi$ on the mod $p$ special fibers. Consider the Stein factorization 
$$\pi_{\rm HT,0}\colon \overline{S}_0^\Sigma(p^n)^{\rm red} \to Y \to    P_-\backslash H_\kappa $$with $\overline{S}_0^\Sigma(p^n)^{\rm red} \to Y$ surjective with connected fibers and   $ Y  \to    P_-\backslash H_\kappa $ finite. 
 Over $ \overline{S}_0^\Sigma(p^n)^{\rm red}$  the sheaf $L_0$ is the pull-back of the tautological invertible sheaf $L^{\rm taut}$ on $P_-\backslash H_\kappa$ via $\pi_{\rm HT,0}$. Hence the restriction of $\Psi_0$ to $\overline{S}_0^\Sigma(p^n)^{\rm red}$  factors via $Y$. Let $Y'\subset Y$ be an irreducible component dominating an irreducible component of $D_0^\Sigma(p^n)$; it must have dimension $\geq d$. 

The scheme $P_-\backslash H_\kappa$ is irreducible and has dimension equal to the dimension of the unipotent radical of $P$, which is also $d$.  As the map $Y' \to P_-\backslash H_\kappa$ is finite, it must be surjective. Hence the map $\pi_{\rm HT,0}\colon \overline{S}_0^\Sigma(p^n)^{\rm red} \to     P_-\backslash H_\kappa$ is surjective and  the claim follows.

\end{proof}

\section{Proof of Theorem \ref{thm:main}}

Thanks to Corollary \ref{cor:piHT0} the two maps of topological spaces $q_1\colon  S^{\rm ad}(p^\infty) \to S_0^\Sigma$ and $q_2\colon S^{\rm ad}(p^\infty) \to \mathrm{F}\ell_{G,\mu^{-1},0}$ factor through a morphism $S^{\rm ad}(p^\infty)\to \overline{S}_0^\Sigma(p^n)$,  the projective morphism  $u \colon   \overline{S}_0^\Sigma(p^n)  \to S_0^\Sigma$ and the map $\pi_{\rm HT,0}\colon \bigl( \overline{S}_0^\Sigma(p^n)\bigr)^{\rm red} \to \mathrm{F}\ell_{G,\mu^{-1},0}\cong P_-\backslash H_\kappa$. Hence it suffices to prove the result for $\overline{\FF}_p$-points of $ \overline{S}_0^\Sigma(p^n)$. Consider the two maps

$$\zeta_1'\colon  \bigl( \overline{S}_k^\Sigma(p^n)\bigr)^{\rm red} \stackrel{u}{\longrightarrow} S_k ^\Sigma\stackrel{\zeta^\Sigma}{\longrightarrow}H-\hbox{{\rm Zip}}^\chi \cong [E_{\mathcal{Z}_{\rm EO}^\chi} \backslash H_k]$$and
$$\zeta_2'\colon  \bigl( \overline{S}_k^\Sigma(p^n)\bigr)^{\rm red} \stackrel{\pi_{\rm HT,0}}{\longrightarrow} P_-\backslash H_k \stackrel{\overline{\gamma}}{\longrightarrow}  \bigl(E_{\mathcal{Z}_{\rm DL}^{\chi^{-1}}} \backslash H_k\bigr),$$where $\overline{\gamma}$ is defined in (\ref{eq:gammabar}) and the index $k$ in $ \overline{S}_k^\Sigma(p^n)$ denotes the base change to $k$.

\subsection{The setup}\label{sec:setup} Take a closed point $x_0$ of $S_0^\Sigma(p^n)$. Suppose it is defined over a finite extension $\FF$ of $\kappa$. First lift the image $x_0'$ of $x_0$ in $S_0^\Sigma$ to a $\WW(\FF)$-valued point $x'$ of $S^\Sigma$ such that the induced $\WW(\FF)[p^{-1}]$-point lies in $S$: this can be done as $\cS^\Sigma$ is log-smooth over $\WW(\kappa)$.  Then lift $x'$, viewd as a $\WW(\FF)$-point of $\cS^\Sigma$  to a $\cO_K$-valued point $x$ of $\overline{\cS}^\Sigma(p^n)$ for some finite field extension $\WW(\kappa)[1/p] \subset K$. Let $\WW(\kappa)[1/p] \subset K_0 \subset K$ be the maximal unramified extension. We will assume that $\FF$ is the residue field of $K_0$.   This gives, in particular, 

\begin{itemize}

\item[i.]  an abelian variety $A$ over $K_0$, with semistable model $\widetilde{A}$ over $\WW(\FF)$ and special fiber $\widetilde{A}_\FF$ such that the Galois group of $K$ acts trivially on $T_pA/p^n T_pA$;

\item[ii.]  elements $(s_{\rm et,\alpha})_\alpha$ in the tensor algebra of \'etale cohomology of $A$, identified with the dual of the $p$-adic Tate module $T_p(A)$ of $A$, fixed by the Galois group of $K_0$; 

\item[iii.] a trivilization $d\colon T_p(A)^\vee\cong \Lambda$, sending the tensors $(s_{\rm et,\alpha})_\alpha$ to the elements $(s_\alpha)_\alpha$;

\item[iv] a finite and free $\WW(\FF)$-module $H_0:=(x')^\ast(\cH)$, pull back of the log crystal $\cH$ defined in \S\ref{sec:boundary} and identified with the  contravariant log Dieudonn\'e module $H_0(\widetilde{A}_\FF)$ of $\widetilde{A}_\FF$ via \cite[Prop. 1.3.5]{MP}, endowed with a Frobenius linear operator $\varphi$ and  elements $(s_{\rm cris,\alpha})_\alpha$ in the tensor algebra of $H_0$, fixed by $\varphi$;

\item[v] a saturated $\WW(\FF)$-submodule $\omega_0\subset H_0$  such that  $\varphi(\omega_0)\subset p H_0$ and a lattice $\omega^{\rm mod}:=(x')^\ast\bigl(\Omega^{\rm mod}\bigr)\subset \omega_0 \otimes_{\WW(\FF)} \cO_K$, pull-back of the sheaf $\Omega^{\rm mod}$ of Proposition \ref{prop:HT1},  such that the Hodge-Tate map $T_p(A^\vee)\otimes \cO_{\CC_p} \to \omega_0 \otimes_{\WW(\FF)} \cO_{\CC_p} $ factors via a surjective map $${\rm HT}\colon T_p(A^\vee)\otimes \cO_{\CC_p}\longrightarrow  \omega^{\rm mod} \otimes_{\cO_K} \cO_{\CC_p};$$

\item[vi.] a trivilization  $q\colon H_0\cong \Lambda \otimes \WW(\FF)$ such that $\omega_0$ is mapped isomorphically onto $\Lambda_0\otimes_{\WW(\kappa)} \WW(\FF)$ and it sends  the tensors $(s_{\rm cris,\alpha})_\alpha$  to  the elements $(s_\alpha)_\alpha$;

\end{itemize}

Thanks to \cite[Prop. A.2.2]{MP} we have a  comparison morphism $$\alpha_{\rm st}\colon T_p(A^\vee)\otimes_{\ZZ_p}  \widehat{A}_{\rm st} \longrightarrow H_0\otimes_{\WW(\FF)} \widehat{A}_{\rm st} \cdot \frac{1}{t}.$$The ring $\widehat{A}_{\rm st}$ is a variant of  Fontaine period rings introduced in \cite{Breuil}: it is the $p$-adic completion of the DP-$A_{\rm cris}$-algebra $A_{\rm cris}\langle X \rangle$, it is endowed with an action  of the Galois group of $K_0$, with an $A_{\rm cris}$-linear derivation  $N$ and with a Frobenius $\varphi$, it admits a Galois equivariant map $\vartheta\colon \widehat{A}_{\rm st}\to \cO_{\CC_p}$ restricting to the classical map of Fontaine on $A_{\rm cris}$ and sending $X\mapsto 0$. The element $t\in A_{\rm cris}$ is the classical period of Fontaine. The map $\alpha_{\rm st}$ is a map of $ \widehat{A}_{\rm st}$-modules, compatible with all these extra structures.  By \cite[Prop 1.4.10 \& Cor. 3.3.9]{MP} it is  an isomorphism after  inverting $t$  and it sends $(s_{\rm et,\alpha})_\alpha$ to $(s_{\rm cris,\alpha})_\alpha$.

Write $\widehat{B}_{\rm st}:=\widehat{A}_{\rm st}[t^{-1}]$. There is a Galois equivariant map $\widehat{A}_{\rm st}\to B_{\rm dR}^+$ (see \cite[\S7]{Breuil}). Consider the base change  $\alpha_{\rm dR}^+(A)$ of $\alpha_{\rm st}(A)$ to $B_{\rm dR}^+$. By loc.~cit., it coincides with the classical de Rham comparison isomorphism of Fontaine after inverting $t$. In particular, 

\begin{corollary}\label{corollary:alphaHT} The reduction of $\alpha_{\rm st}$ via  $\widehat{A}_{\rm st}\to \cO_{\CC_p}$  is the map $d\log\colon T_pA^\vee\otimes_{\ZZ_p} \cO_{\CC_p} \to \Omega_{\widetilde{A}} \otimes_{\WW(\FF)} \cO_{\CC_p}$ of Proposition \ref{prop:HT1}.  Moreover, $H_0\otimes_{\WW(\FF)} \widehat{A}_{\rm st} $ is contained in the image of $\alpha_{\rm st}$. 
\end{corollary}
\begin{proof} Consider the map $\alpha_{\rm st}(A^\vee)$ for the dual abelian variety, $$\alpha_{\rm st}(A^\vee)\colon T_pA \otimes_{\ZZ_p}  \widehat{A}_{\rm st} \longrightarrow H_0(\widetilde{A}_\FF^\vee)\otimes_{\WW(\FF)} \widehat{A}_{\rm st} \cdot \frac{1}{t}.$$
Then $H_0(\widetilde{A}^\vee_\FF)\cong H_0(\widetilde{A}_\FF)^\vee=H_0^\vee$ by \cite[Prop. 1.3.5]{MP} and $(T_pA)^\vee \cdot t  \cong T_p(A^\vee) $. Taking the $\widehat{A}_{\rm st}$-dual of $\alpha_{\rm st}(A^\vee)$ and using these identifications we get a map 

\begin{equation}\label{eq:betast} \beta_{\rm st} \colon H_0\otimes_{\WW(\FF)} \widehat{A}_{\rm st} \longrightarrow T_p(A^\vee)\otimes_{\ZZ_p}  \widehat{A}_{\rm st}.\end{equation}

We claim that $\beta_{\rm st} \circ \alpha_{\rm st} $ is multiplication by $t$. This can be proved after inverting $t$. Notice that by \cite[Prop. A.2.7]{MP} the base change of the maps $\alpha_{\rm st} $ and $\beta_{\rm st} $ to $\widehat{B}_{\rm st}$ are already defined over the subring $B_{\rm st }\subset \widehat{B}_{\rm st}$. Then the fact that $\beta_{\rm st} \circ \alpha_{\rm st} $ is multiplication by $t$ can be checked after base change to  $B_{\rm dR}$, as the induced map   $B_{\rm st } \to B_{\rm dR}$ is injective.  The base change of $\alpha_{\rm st}$ and $\beta_{\rm st}$ to $B_{\rm dR}$ are Fontaine's classical de Rham comparsion isomorphisms. The claim follows then from the fact that the  de Rham comparison isomorphsims are compatible with the Poincar\'e duality in \'etale and de Rham cohomology respectively. 

\end{proof}

\subsection{A description of $\zeta_1'(x_0)$ and $\zeta_2'(x_0)$}

Fix a basis $\mathcal{C}$  of $\Lambda$ as $\ZZ_p$-module. Define $g\in G_{\ZZ_p}(\WW(\FF))$  to be $$g:=q \circ  \frac{\varphi}{p}\vert_{\omega_0} \oplus \varphi\vert_{q^{-1}(\Lambda_{-1}\otimes_{\WW(\kappa)} \WW(\FF))} \circ q^{-1}\colon \Lambda\otimes_{\WW(\kappa)} \WW(\FF)) \longrightarrow \Lambda\otimes_{\WW(\kappa)} \WW(\FF)).$$Here $q$ is the trivialization in \S\ref{sec:setup}(vi). Let $g_0\in H(\FF)$ to be the reduction of $g$ modulo $p$; recall that $H=(G_{\ZZ_p})_{\FF_p}$. As explained in \S\ref{sec:EO}  $$\zeta_1'(x_0)=[g_0]\in E_{\mathcal{Z}_{\rm EO}}(k) \backslash H(k).$$ Assume $\zeta_1'(x_0)\in H_{w \widetilde{z}^{-1}}$ for some $w\in {^I W} $ and  with  $z:=w_0^K$, see Lemma \ref{lemma:z} for the notation. Thanks to Proposition \ref{prop:LieE} there exists a lift $x'$ such that,  possibly after multiplying $g$ by an element $\gamma=(a,b^{(p)})$ of $\widetilde{E}_{\mathcal{Z}_{\rm EO}}(\WW(k))$ with $a\in \widetilde{P}(\WW(k))$ and $b\in \widetilde{P}_-(\WW(k))$, we may assume that  $w\widetilde{z}^{-1}=a g (b^{(p)})^{-1}$.

 We denote by $\mathcal{F}$ the basis $q^{-1}(\mathcal{C})$ of $H_0$.  Write $a=u_+ \ell $ and $b^{(p)}=u_- \ell^{(p)}$ for their decomposition using $\widetilde{P}= \mathrm{R}_u \widetilde{P} \cdot \widetilde{L}$ and $\widetilde{Q}= \mathrm{R}_u \widetilde{Q} \cdot \widetilde{L}^{(p)} $.  We let $\mathcal{A}$  be the basis of $H_0$ such that the matrix of change of basis is ${_{{\mathcal{A}}}} M_{\mathcal{F}}  ({\rm Id})=a^{-1}$. 
Thanks to Lemma \ref{lemma:zetaSigma} we have  $${_\mathcal{A}} M_{\mathcal{A}^{(p)}} (\varphi)= a g \left(\begin{matrix} p & 0 \cr 0 & 1 \cr \end{matrix}\right) \bigl(a^{(p)}\bigr)^{-1} =a g  \bigl(\ell^{(p)}\bigr)^{-1} \left(\begin{matrix} p & 0 \cr 0 & 1 \cr \end{matrix}\right) \bigl(u_+^{(p)}\bigr)^{-1}=w\widetilde{z}^{-1}   u_- \left(\begin{matrix} p & 0 \cr 0 & 1 \cr \end{matrix}\right) \bigl(u_+^{(p)}\bigr)^{-1}  ,$$as $ \left(\begin{matrix} p & 0 \cr 0 & 1 \cr \end{matrix}\right)$ commutes with $\widetilde{L}^{(p)}$, and $w\widetilde{z}^{-1}=a g \bigl(\ell^{(p)}\bigr)^{-1} u_-^{-1}$.
We remark that
$r_-:=u_- \left(\begin{matrix} p & 0 \cr 0 & 1 \cr \end{matrix}\right) \in {^{\widetilde{z}}} \widetilde{B}\bigl(\WW(k)[p^{-1}] \bigr)$ and that $u_+ \in \mathrm{R}_u  \widetilde{P}(\WW(k))$ as $\left(\begin{matrix} p & 0 \cr 0 & 1 \cr \end{matrix}\right)\in \widetilde{T}\bigl(\WW(k)[p^{-1}] \bigr)$ and $u_-\in  \mathrm{R}_u \widetilde{P}_-^{(p)} \bigl(\WW(k) \bigr)\subset  {^{\widetilde{z}}}  \mathrm{R}_u \widetilde{B}\bigl(\WW(k) \bigr)$. In particular,  we have the following formula:

\begin{equation}\label{eq:Frobmatrix}  {_\mathcal{A}} M_{\mathcal{A}^{(p)}} (\varphi)  =w\widetilde{z}^{-1} \cdot  r_- \cdot u_+^{(p)}, \quad r_-\in {^{\widetilde{z}}} \widetilde{B}\bigl(\widehat{A}_{\rm st}\bigl[p^{-1}\bigr]\bigr), \quad  u_+ \in \mathrm{R}_u\widetilde{P}(\widehat{A}_{\rm st}).
\end{equation}

\

On the other hand we have a surjective map $${\rm HT} \circ d^{-1} \colon \Lambda \otimes \cO_{\CC_p}(1)  \longrightarrow   T_p A^\vee \otimes \cO_{\CC_p} \longrightarrow \omega^{\rm mod}\otimes_{\cO_K} \cO_{\CC_p}$$defining the point $\pi_{\rm HT}(x)$, where $d$ is the trivialization in \ref{sec:setup}(iii). Take $h\in G_{\ZZ_p}(\cO_{\CC_p})$ to be an element such that ${\rm Ker}({\rm HT} \circ d^{-1})$ is  the filtration $h (\Lambda_{-1} \otimes_{\WW(\FF)} \cO_{\CC_p}) $. Then $h_0$, which is $h$ modulo $p$, represents  $\pi_{\rm HT,0}(x_0)$ and $$\zeta_2'(x_0)=[h_0 (h_0^{(p)})^{-1}]\in  E_{\mathcal{Z}_{\rm opp}}(k) \backslash H(k).$$We let $v\in {^{I_-} W}$ be such that $\zeta_2'(x_0)$  lies in the $v$-stratum  as defined in Corollary \ref{cor:piHT0} and the  discussion that follows the Corollary, that is, $$v=\widetilde{y}^{-1} h_0 (h_0^{(p)})^{-1} \widetilde{z} \in E_{\mathcal{Z}_-}(k) \backslash H(k) \cong {^{I_-} W}.$$Recall that   $y:={^I}w_0=  w_0^{I_-}$ and, denoting by $\varphi$  the automorphism of  $W$ defined by Frobenius on $G_k$, we have $\varphi(y)=z$. 


Denote by  $\mathcal{E}=d^{-1}(\mathcal{C})$  the basis of $T_p(A)^\vee$ given by pull back via $d$ of the given basis $\mathcal{C}$ of $\Lambda$. Lift $h\in G(\cO_{\CC_p})$ to an element $\widetilde{h}\in G\bigl(\widehat{A}_{\rm st}\bigr)$ via the surjective map $\vartheta\colon \widehat{A}_{\rm st} \to \cO_{\CC_p}$. Then $\widetilde{h} (\widetilde{h}^{(p)})^{-1}\in G\bigl(\widehat{A}_{\rm st}\bigr)$ lifts $h_0 (h_0^{(p)})^{-1}$. Here $\widetilde{h}^{(p)}$ is defined using the Frobenius $\varphi$ on $\widehat{A}_{\rm st}$.  Let $\mathcal{D}$ be the basis of $T_p(A)^\vee \otimes \widehat{A}_{\rm st} $ defined by $\widetilde{h}^{-1} \cdot d$. Then the matrix of change of basis expressing $\mathcal{E}$ as a combination of $\mathcal{D}$ is ${_{{\mathcal{D}}}} M_{\mathcal{E}}  ({\rm Id})=\widetilde{h}$.  Let Frobenius $\varphi$ on $T_p(A) \otimes \widehat{A}_{\rm st} $ be $1\otimes \varphi$. Its matrix with respect to the basis $\mathcal{E}$  is the identity so that 

\begin{equation}\label{eq:Frobeniush} {_{{\mathcal{D}}}} M_{\mathcal{D}{^{(p)}}}  (\varphi)={_{{\mathcal{D}}}} M_{\mathcal{E}}  ({\rm Id}) {_{{\mathcal{E}}}} M_{\mathcal{E}{^{(p)}}}  (\varphi) {_{\mathcal{E}^{(p)}}} M_{\mathcal{D}{^{(p)}}}  ({\rm Id}) = \widetilde{h} (\widetilde{h}^{(p)})^{-1}.\end{equation}

\subsection{Reduction steps}

We let  $\tau\in G(\widehat{B}_{\rm st})$ be the element $\tau={_{\mathcal{A}}} M_{\mathcal{D}}(\alpha_{\rm st})$. It is the matrix of $$((a q) \otimes 1) \circ  \alpha_{\rm st}  \circ   ((\widetilde{h}^{-1} d) \otimes 1)^{-1} \colon \Lambda \otimes \widehat{A}_{\rm st} \longrightarrow T_p(A^\vee)\otimes_{\ZZ_p}  \widehat{A}_{\rm st}  \longrightarrow  H_0\otimes_{\WW(\FF)} \widehat{A}_{\rm st}  \longrightarrow \Lambda \otimes \otimes_{\ZZ_p}  \widehat{A}_{\rm st}$$with respect to the given  basis $\mathcal{C}$ of $\Lambda$. Then,

\begin{proposition}\label{prop:decompostau} Assume that $p>3$. There exists a $\widehat{A}_{\rm st} $-algebra $R^+$ with the following properties:

\begin{itemize}
\item[i.] $R^+$ is $t$-torsion free and local with maximal ideal $\mathfrak{m}$ and residue field $k$ and   the quotient map $j\colon R^+\to k$ is compatible with the map $\widehat{A}_{\rm st} \to k$ given by composing $\vartheta_{\rm st}\colon\widehat{A}_{\rm st}  \to \cO_{\CC_p}$ and the residue field map $\cO_{\CC_p}\to k$. We write $R:=R^+[t^{-1}]$;

\item[ii.] $\tau$ and $\tau^{(p)}$, viewed in $G(R)$, admit decompositions $\tau=\beta \cdot   \ell  $ and  $\tau^{(p)}=\beta' \cdot \ell' $ such that $\beta\in  {^{\widetilde{y}}}\widetilde{B}(R)$ and $\beta'\in {^{\widetilde{z}}}\widetilde{B}(R)$ and  $\ell$ and $\ell'$ are in $G(R^+)$ and their reduction $\overline{\ell}$ and $\overline{\ell}'$ modulo $\mathfrak{m}$ lie in $L(k)$ and satisfy $\overline{\ell}'=\overline{\ell}^{(p)}$.

\item[iii.] we have a decomposition as in (ii) with the same $\overline {\ell}$ of (ii), if we replace $\tau$ and $\tau^{(p)}$ with $s \cdot \tau$, resp.~$(s \cdot \tau)^{(p)}$, for some $s\in \mathrm{R}_u \widetilde{P}(\widehat{A}_{\rm st})$.

\end{itemize}

\end{proposition}
\begin{proof}
We write $\tau$ as $$\tau=\left(\begin{matrix}  a & b \cr c & d \cr \end{matrix} \right)\in {\rm End}( \Lambda) \otimes \widehat{A}_{\rm st}$$ with $a\in {\rm End}( \Lambda_0) \otimes_{\WW(\kappa)} \widehat{A}_{\rm st}$, $b\in {\rm Hom} ( \Lambda_1, \Lambda_0)\otimes_{\WW(\kappa)} \widehat{A}_{\rm st}$ etc.

Due to Corollary \ref{corollary:alphaHT} we have that $b$, $c$, $d$ are $0$ modulo ${\rm Ker}(\vartheta_{\rm st})$ and if we write  $\delta=\left(\begin{matrix}  m & n \cr u & v \cr \end{matrix} \right)\in {\rm End}( \Lambda \otimes \widehat{A}_{\rm st})$ 
for the matrix ${_{\mathcal{D}}} M_{\mathcal{A}}(\beta_{\rm st})$ of the morphism $\beta_{\rm st}$ of (\ref{eq:betast}) we have $$\delta \cdot \tau= \tau \cdot \delta= t \mathbf{1}_{\Lambda}.$$Then also $m$, $n$ and $u$ lie in ${\rm Ker}(\vartheta_{\rm st})$.

Consider the ring $A_{0,p-1,\omega}^{\rm st}$ defined in \S\ref{sec:Appendix}. It is a  $\widehat{A}_{\rm st}$-algebra, the map $\vartheta_{\rm st}$ extends to a map $\vartheta_{0,p-1,\omega}\colon A_{0,p-1,\omega}\to \cO_{\CC_p}$ and  ${\rm Ker}(\vartheta_{\rm st}) \subset  \omega A_{0,p-1,\omega}^{\rm st}$. The element $t\in A_{\rm cris}$ is equal to $[\varepsilon]-1$ times a unit of $A_{\rm cris}$ (see for example \cite[Lemme 6.2.13]{olivier}). As $[\varepsilon]-1=\bigl([\varepsilon^{\frac{1}{p}}]-1\bigr) \omega$ times a unit of  $A_{\rm inf}$ we conclude that $t=\bigl([\varepsilon^{\frac{1}{p}}]-1\bigr) \omega g$ with $g$ a unit of $A_{\rm cris}$. 
Hence

$$\tau':= \widetilde{\mu}(\omega)^{-1}  \tau  =\left(\begin{matrix}  a & b   \cr c \omega^{-1}  & d \omega^{-1}\cr \end{matrix} \right) ,\quad \delta':= \delta \cdot \bigl( \widetilde{\mu}(\omega) \omega^{-1}\bigr) =\left(\begin{matrix}  m\omega^{-1} & n \omega^{-1} \cr u & v \cr \end{matrix} \right)$$ are both elements of ${\rm End}( \Lambda) \otimes A_{0,p-1,\omega}^{\rm st}$ and $$ \delta' \cdot \tau' \equiv [\underline{\pi}^{\frac{1}{p-1}}] \mathbf{1}_{2g}$$modulo ${\rm Ker}(\vartheta^{\rm st}_{0,p-1,\omega}) $.

As $\vartheta_{0,p-1,\omega}(b)=\vartheta_{0,p-1,\omega}(n)=0$,  the reduction $\overline{\tau}'$ of $\tau'$ modulo ${\rm Ker}(\vartheta_{0,p-1,\omega})$ lies in $P_-(\CC_p)$. Thanks to Lemma \ref{lemma:approximate} we can find $\tau_0\in \widetilde{L}(A_{\rm inf})$ and $\tau_1\in  {^{\widetilde{y}}}\widetilde{B}\bigl(A_{\rm inf}\bigl[[\underline{p}]^{-1}\bigr]\bigr) \cap {\rm End}(\Lambda) \otimes A_{\rm inf}$ such that $\tau' \equiv \tau_1 \tau_0$ in $ {\rm End}(\Lambda) \otimes \cO_{\CC_p}$ (via the map $\vartheta^{\rm st}_{0,p-1,\omega}$). We define $\delta_0$ and $\delta_1$ simillarly.  Then $$ (\delta_1 \delta_0) \cdot (\tau_1\tau_0) \equiv  [\underline{\pi}^{\frac{1}{p-1}}] \mathbf{1}_{2g} $$ modulo ${\rm Ker}(\vartheta) $. We now pass to $A_{r,p-1,\omega}^{\rm st}$, for $1 > r> \frac{1}{p-1}$. It is a $A_{0,p-1,\omega}^{\rm st}$-algebra and $\vartheta_{0,p-1,\omega}$ extends to the morphism $\vartheta_{r,p-1,\omega}$ by Lemma \ref{lemma:Astrs}. Corollary \ref{corollary:invertiblematrix} implies that 

\begin{equation}\label{eq:tau'}  (\delta_1 \delta_0) \cdot (\tau_1 \tau_0)  =  [\underline{\pi}^{\frac{1}{p-1}}] D \end{equation} with $D\in G(A_{r,p-1,\omega}^{\rm st})$.

Let us now work with $A_{r,s,[\underline{\pi}]^r\omega}^{\rm st}$ with $s = p-1- r$. Using Lemma \ref{lemma:Astrs} again, we know that it is a $A_{r,p-1,\omega}^{\rm st}$-algebra and $\vartheta_{r,p-1,\omega}$ extends to the morphism $\vartheta_{r,s,[\underline{\pi}]^r\omega}$. Moreover,  ${\rm Ker}(\vartheta_{0,p-1,\omega}) \subset  [\underline{\pi}]^r A_{r,s,\omega}^{\rm st}$ by Lemma \ref{lemma:Astrs}. Hence, $\tau' \equiv \tau_1 \tau_0$ modulo $ [\underline{\pi}]^r A_{r,s,[\underline{\pi}]^r\omega}^{\rm st}$. Consider the matrix $$\tau'':= (\tau_1\tau_0)^{-1} \tau'= \left(\begin{matrix}  a' & b' \cr c' & d' \cr \end{matrix} \right) \in G_{\ZZ_p}\bigl(A_{r,s,[\underline{\pi}]^r\omega}^{\rm st} [\underline{\pi}]^{-1}\bigr)\cap {\rm End}(\Lambda)\otimes A_{r,p-1,\omega}^{\rm st}\bigl[ [\underline{\pi}]^{-1}\bigr] .$$Due to (\ref{eq:tau'}) we have that $(\tau_1\tau_0)^{-1}$ is in $ {\rm End}(\Lambda)\otimes \left( [\underline{\pi}^{\frac{-1}{p-1}}] \cdot A_{r,s,[\underline{\pi}]^r\omega}^{\rm st}\right)$ so that  $[\underline{\pi}^{\frac{1}{p-1}}] \left(\begin{matrix}  a' & b' \cr c' & d' \cr \end{matrix} \right)$ is in ${\rm End}(\Lambda)\otimes A_{r,s,[\underline{\pi}]^r\omega}^{\rm st}$ and is congruent to $ [\underline{\pi}^{\frac{1}{p-1}}] \mathbf{1}_{2g}$ modulo $[\underline{\pi}]^r A_{r,s,[\underline{\pi}]^r\omega}^{\rm st}$.  Hence $\tau''$ is a matrix congruent to $\mathbf{1}_{2g}$ modulo  $ [\underline{\pi}]^{r-\frac{1}{p-1}}$ and, hence, it is invertible as $A_{r,s,[\underline{\pi}]^r\omega}^{\rm st}$ is $[\underline{\pi}]$-adically complete and separated by Lemma \ref{lemma:Astrs}, that is $\tau''\in G\bigl(A_{r,s,[\underline{\pi}]^r\omega}^{\rm st}\bigr)$.

In conclusion, we have $\tau=\beta \ell$ with $\beta:=\widetilde{\mu}(\omega)  \tau_1 \in   {^{\widetilde{y}}} \widetilde{B}_-\bigl(A_{r,s,[\underline{\pi}]^r\omega}^{\rm st}\bigl[[\underline{p}]^{-1} \cdot \omega^{-1}\bigr]\bigr)$ and $\ell:=\tau_0 \cdot \tau'' \in G_{\ZZ_p}\bigl(A_{r,s,[\underline{\pi}]^r\omega}^{\rm st}\bigr)$ such that $\overline{\ell}\equiv \overline{\tau}_0 \in L(k)$; here $\overline{\ell}$, resp.~$\overline{\tau}_0$, is the image of $\ell$, resp.~$\tau_0$ in $G_{\ZZ_p}(k)=H(k)$ via the map $j_{r,s}\colon A_{r,s,[\underline{\pi}]^r\omega}^{\rm st}\to k$ defined by $\vartheta_{r,s,[\underline{\pi}]^r\omega}$ composed with the quotient $\cO_{\CC_p}\to k$.

Frobenius $\varphi$ on $\widehat{A}_{\rm st}$ extends to a map  $\varphi\colon A_{r,s,[\underline{\pi}]^{r}\omega}^{\rm st} \longrightarrow  A_{pr,s,[\underline{\pi}]^{pr}\varphi(\omega)}^{\rm st}$ (see lemma \ref{lemma:Astrs}). The latter is well defined, together with the map $j_{pr,s}\colon A_{pr,s,[\underline{\pi}]^{pr}\varphi(\omega)}^{\rm st} \to k$, if and only if  $pr < s=(p-1) - r$, i.e., if there exists $r > \frac{1}{p-1}$ such that $(p+1) r < p-1$ which is equivalent to require that $p>3$ and explains the assumption. Then $$\tau^{(p)}=\beta' \ell' ,\qquad \beta'=\varphi(\beta)  \in {^{\widetilde{z}}}\widetilde{B}\bigl(A_{pr,s,[\underline{\pi}]^{pr}\varphi(\omega)}^{\rm st}\bigl[[\underline{p}]^{-1} \cdot \varphi(\omega)^{-1}\bigr]\bigr) , \quad \ell'=\varphi(\ell) \in G_{\ZZ_p}\bigl(A_{r,s,[\underline{\pi}]^r\omega}^{\rm st}\bigr).$$By construction $\overline{\ell}'=\overline{\ell}^{(p)}\in L^{(p)}(k)$. Here $\varphi(\beta) $ and $\varphi(\ell)$ are defined by applying $\varphi$ to the entries of  the matrices $\beta$ and $\ell$. They lie in  $ {^{\widetilde{z}}}\widetilde{B}$ as $\varphi(y)=z$ and in $G_{\ZZ_p}$ respectively as  they are both subgroups of ${\rm GSp}(\Lambda)$ defined over $\ZZ_p$.

Finally $\omega= [\underline{\pi}] \bigl( [\underline{\pi}]^{pr-1} \frac{p}{[\underline{\pi}]^{pr}} -1)$ is equal to $ [\underline{\pi}]$ times a unit in  $ A_{pr,s,[\underline{\pi}]^{pr}\varphi(\omega)}^{\rm st}$. Analogously $\varphi(\omega)=p \bigl(1 - \frac{ [\underline{\pi}]^s}{p}  [\underline{\pi}]^{p-s}\bigr)$ is $p$ times a unit in  $A_{r,s,[\underline{\pi}]^{r}\omega}^{\rm st}$. Hence both $ A_{r,s,[\underline{\pi}]^{r}\omega}^{\rm st}$ and $ A_{pr,s,[\underline{\pi}]^{pr}\varphi(\omega)}^{\rm st}$ 
map  to $ A_{pr,s,[\underline{\pi}]^{pr}p}^{\rm st} $. We let $R^+$ be the localization of  $A_{pr,s,[\underline{\pi}]^{pr}p}^{\rm st}$ with respect to the maximal ideal defined by the kernel of the map $j_{pr,s}\colon A_{pr,s,[\underline{\pi}]^{pr}p}^{\rm st} \to k$. As $t=\bigl([\varepsilon^{\frac{1}{p}}]-1\bigr) \omega u$ with $u$ a unit in $R^+$ and $\omega= [\underline{\pi}]$ times a unit in $R^+$, inverting $t$ or inverting $  [\underline{\pi}]$ in $R^+$ gives the same ring $R$. As $p= [\underline{\pi}]^{p-1}=p \cdot \frac{[\underline{\pi}]^{p-1}}{p}$ then $p$ is invertible in $R$ as well.  As $R^+$ is $p$-torsion free and $p\in [\underline{\pi}]^{pr} A_{pr,s,[\underline{\pi}]^{pr}p}^{\rm st} $, then $R^+$ is also $[\underline{\pi}]$-torsion free.

For (iii) the same argument for $s \cdot \tau$ in place of $\tau$ holds. We only need to prove that one obtains a decomposition with the  same $\overline{\ell}$.  Define $$\tau''':=\widetilde{\mu}(\omega)^{-1} s\cdot \tau =\bigl(\widetilde{\mu}(\omega)^{-1}  s \widetilde{\mu}(\omega) \bigr) \bigl( \widetilde{\mu}(\omega)^{-1} \tau\bigr).$$Then $s'=\widetilde{\mu}(\omega)^{-1}  s \widetilde{\mu}(\omega) $ is congruent to $\mathbf{1}_{2g}$ modulo $\omega$ as $s\in  \mathrm{R}_u \widetilde{P} (\widehat{A}_{\rm st})$. Hence,  $\tau''' \equiv \tau_1 \tau_0$ in $ {\rm End}(\Lambda) \otimes \cO_{\CC_p}$. Then the argument continues as before and we get a decomposition $s \cdot \tau=\beta'' \ell''$ with  $\beta''\in {^{\widetilde{y}}}\widetilde{B}(R)$ and $\ell''\in G(R^+)$ and $\ell''\equiv \overline{\tau}_0\equiv \ell_0$ modulo $\mathfrak{m}$ as wanted. 

\end{proof}

\begin{remark} If one is interested in a decomposition of $\tau$ as in Proposition \ref{prop:decompostau}, one can work with $B_{\rm dR}$ and the proof is much easier. It is the need to decompose $\tau^{(p)}$ as well that forced us to look into the finer structure of $\widehat{A}_{\rm st}$. As we need the existence of a positive rational number $r$ such that $r > \frac{1}{p-1}$ and $(p+1) r < p-1$ the hypothesis $p\geq 5$ is forced on us.
\end{remark}

\subsection{Conclusion}

Recall that $\tau={_{\mathcal{A}}} M_{\mathcal{D}}(\alpha_{\rm st})$. It follows from (\ref{eq:Frobmatrix}) and (\ref{eq:Frobeniush}) that we have the following equality in $G(R)$:  $$\tau  \widetilde{h}(\widetilde{h}^{(p)})^{-1} \bigl((u_+ \cdot \tau)^{(p)}\bigr)^{-1}= w  \widetilde{z}^{-1} r_-.$$Decompose $\tau=\beta \ell$ and $u_+ \cdot \tau=\rho q $ as in Proposition \ref{prop:decompostau} with $\ell$ and $q\in G_{\ZZ_p}(R^+)$ such that  $\overline{\ell}=\overline{q}=:\ell_0\in L(k)$ modulo $\mathfrak{m}$ and $\beta$ and $\rho\in {^y}\widetilde{B}(R)$.  Consider the class 

$$\gamma:=\ell \widetilde{h}(\widetilde{h}^{(p)})^{-1} \bigl(q^{(p)} \bigr)^{-1} \in {^{\widetilde{y}}}\widetilde{B}(R^+)\backslash G(R^+) /{^{\widetilde{z}}}\widetilde{B}(R^+).$$Its image $\gamma_0$ via  $j\colon R^+\to k$ is $\gamma_0=\ell_0 h_0(h_0^{(p)})^{-1} (\ell_0^{(p)})^{-1}  \in {^{\widetilde{y}}}B(k)\backslash G(k) /{^{\widetilde{z}}}B(k)$. We also have $$\gamma\equiv \beta \gamma \bigl(\rho^{(p)}\bigr)^{-1} = \tau  \widetilde{h}(\widetilde{h}^{(p)})^{-1} \bigl((u_+ \cdot \tau)^{(p)}\bigr)^{-1}= w  \widetilde{z}^{-1} r_- = w  \widetilde{z}^{-1}$$in ${^{\widetilde{y}}}\widetilde{B}(R)\backslash G(R) / {^{\widetilde{z}}}\widetilde{B}(R)$ as $r_-\in {^{\widetilde{z}}} \widetilde{B}(R)$ by (\ref{eq:Frobmatrix}). 
In conclusion, $$\gamma_0=\ell_0 h_0(h_0^{(p)})^{-1} (\ell_0^{(p)})^{-1} \leq_-  w  \widetilde{z}^{-1} $$for the Bruhat order $\leq_-$ on  $ G /{^{\widetilde{z}}}\widetilde{B}$ (with respect to the left action of ${^{\widetilde{y}}}\widetilde{B}$).

Recall that we have chosen $h_0$ so that $\widetilde{y}^{-1}  h_0(h_0^{(p)})^{-1} \widetilde{z}=v \in {^{I_-}}W$. It  is the same class as  $\widetilde{y}^{-1} \gamma_0  \widetilde{z}$ for the order relation $\preccurlyeq$ on ${^{I_-}}W$. Consider the  isomorphism $$\delta\colon  G /{^{\widetilde{z}}}\widetilde{B} \to  G /\widetilde{B}, [g] \mapsto \bigl[\widetilde{y}^{-1} g \widetilde{z}\bigr];$$it is equivariant with respect to the left action of ${^{\widetilde{y}}}\widetilde{B}$ on the left hand side, resp.~$\widetilde{B}$ on the right hand side. In particular, it is order preserving  considering on  $G /\widetilde{B}$ the standard Bruhat order. 
Hence,   $$\delta\bigl(\ell_0 h_0 h_0^{(p)}(\ell_0^{(p)})^{-1}\bigr) \leq \widetilde{y}^{-1} w  $$for the standard Bruhat order. Recall that the class of $\widetilde{y}$ in the Weyl group $W$ is ${^I w}_0$ so that the class of $\widetilde{y}^{-1}  $ is $w_0^I$. Thanks to Lemma \ref{lemma:dualityonWJ}, the map $$\sigma_0\colon {^I W}  \longrightarrow {^{I_-}} W, \qquad w\mapsto y^{-1}  w$$is an order reversing bijection and we have 
$$ v=\delta\bigl(\ell_0 h_0 h_0^{(p)}(\ell_0^{(p)})^{-1}\bigr) \preccurlyeq \sigma_0(w) $$in ${^{I_-}}W$ as wanted.

\section{Appendix I: Fontaine's rings}\label{sec:Appendix}

Given the tilt $\CC_p^\flat$ of $\CC_p$ and the tilt $\cO_{\CC_p}^\flat$ of $\cO_{\CC_p}$, we have Fontaine's classical ring $A_{\rm cris}$. It is the $p$-adic completion of the DP envelope of $A_{\rm inf}=\WW(\cO_{\CC_p}^\flat)$ with respect to the kernel of the surjective map $\vartheta\colon A_{\rm inf} \to \cO_{\CC_p}$. The ideal ${\rm Ker}(\vartheta)$ is principal. Generators are, for example, $([\epsilon]-1)/([\epsilon^{1/p}]-1)$ or $p-[\underline{p}]$ where $\epsilon=(1,\zeta_p,\zeta_{p^2},\ldots)\in \cO_{\CC_p}^\flat$ is a compatible sequence of primitive $p^n$-th roots of unity and $\underline{p}=(p,p^{1/p}, p^{1/p^2},\ldots)\in \cO_{\CC_p}^\flat$ is a compatible sequence of $p^n$-th roots of $p$. Let $\omega$ be a generator of ${\rm Ker}(\vartheta)$. We also have the ring  $\widehat{A}_{\rm st} := A_{\rm cris}\langle X \rangle$, introduced by C. Breuil in  \cite{Breuil}, given by  the $p$-adic completion of the DP polynomial ring $A_{\rm cris}[X]$. The map $\vartheta$ extends to a map $\vartheta_{\rm st}\colon \widehat{A}_{\rm st} \to \cO_{\CC_p}$ sending $X\mapsto 0$. It is also endowed with a Frobenius, extending the classical one on $A_{\rm cris}$, requiring  that $\varphi(X)=(1+X)^p-1$. 

Given elements $r$ and $s\in\QQ$ with $0 \leq r\leq s$ let $$A_{r,s}=A_{\rm inf}\left\{ \frac{p}{[\underline{p}]^r},\frac{[\underline{p}]^s}{p}\right\},$$be the $p$-adic completion of the ring $A_{\rm inf}[W,Z]/([\underline{p}]^r W-p, pZ-[\underline{p}]^s,WZ-[\underline{p}]^{s-r})$. It coincides with the $p$-adic completion of the subring $A_{\rm inf}\left[ \frac{p}{[\underline{p}]^r},\frac{[\underline{p}]^s}{p}\right]$ of ${\rm Frac}(A_{\rm inf})$ by \cite[Cor. 2.2]{Berger}.  For $\gamma\in A_{r,s}$ a non-zero element we define $$A_{r,s,\gamma}^{\rm st}=A_{\rm inf}\left\{ \frac{p}{[\underline{p}]^r},\frac{[\underline{p}]^s}{p},\frac{X}{p\gamma} \right\}$$to be the $p$-adic completion of the ring $$A_{\rm inf}[W,Z,X,Y]/\bigl([\underline{p}]^r W-p, pZ-[\underline{p}]^s,WZ-[\underline{p}]^{s-r},p\gamma  Y-X\bigr).$$ For $r < 1 <s$  the map $\vartheta_{\rm st}$ estends to a map $$\vartheta_{r,s,\gamma}\colon A_{r,s,\gamma}^{\rm st} \longrightarrow \cO_{\CC_p}, \quad W\mapsto p^{1-r}, Z\mapsto p^{s-1}, Y\mapsto 0.$$
For $r<s$ we also have the map  $$j_{r,s}\colon A_{r,s,\gamma}^{\rm st} \longrightarrow k,$$that on $A_{\rm inf}$ is  the composite of $\vartheta$ with the quotient map $\cO_{\CC_p} \to k$ on the residue field and that sends $W$, $Z$, $X$ and $Y$ to $0$. It coincides with the composite of the quotient map $\cO_{\CC_p} \to k$ with $\vartheta_{r,s,\gamma}$ whenever the latter exists.

\begin{lemma}\label{lemma:Astrs} Assume that $p\geq 3$. Then

\begin{itemize}

\item[1)] For every $s\leq  p-1  $ the map $A_{\rm inf}\to A_{r,s,\gamma}^{\rm st}$ extends to a map $\widehat{A}_{\rm st}\to A_{r,s,\gamma}^{\rm st}$. We have ${\rm Ker}(\vartheta_{\rm st}) \subset [\underline{p}]^{{\rm min}(r,1)}  A_{r,s,\gamma}^{\rm st}$ and ${\rm Ker}(\vartheta_{\rm st}) \subset \omega  A_{r,s,\gamma}^{\rm st}$ if $\gamma\in \omega A_{r,s}$. Moroever, the ring $ A_{r,s,\gamma}^{\rm st}$  is $p$-torsion free and $p$-adically, and hence $[\underline{p}]$-adically, complete and separated. 

\item[2)] For $r < 1 <s$ the kernel of $\vartheta_{r,s,\gamma}$ is generated by $(\omega, W- [\underline{p}]^{1-r}, Z-[\underline{p}]^{s-1}, Y) $.

\item[3)] For $r<r' < 1 < s' <s $ and $\alpha\in A_{r',s'}$ a non-zero element, there are unique morphisms of $A_{\rm inf}$-algebras $A_{r,s}\to A_{r's'}$, resp. of $\widehat{A}_{\rm st}$-algebras  $A_{r,s,\gamma}^{\rm st} \to A_{r',s',\alpha \gamma}^{\rm st}$ sending $W \mapsto [\underline{p}]^{r'-r} W$, $Z\mapsto [\underline{p}]^{s-s'} Z$, $Y\mapsto \alpha Y$. 

\item[4)] For $pr<s$ Frobenius $\varphi$ on $\widehat{A}_{\rm st}$ extends uniquely to a ring homomorphism  $\varphi\colon A_{r,s,\gamma}^{\rm st} \to A_{pr,s, \varphi(\gamma)}^{\rm st}$ such that $\varphi(W)=W $, $\varphi(Z)=[\underline{p}]^{(p-1)s} Z$, $\varphi(Y)=Y \cdot \frac{(1+X)^p-1}{X}$.

\end{itemize}
\end{lemma}
\begin{proof} (1) The ring $A_{r,s}$ is proven to be $p$-torsion free and $p$-adically complete and separated in \cite{Berger}. As $A_{r,s,\gamma}^{\rm st}$ is the $p$-adic completion of 
$A_{r,s}[X,Y]/(p\gamma  Y-X) \cong A_{r,s}[Y]$, also $A_{r,s,\gamma}^{\rm st}$ is $p$-torsion free and  $p$-adically complete and separated.

By definition of $\widehat{A}_{\rm st}$ to prove the first claim it suffices to show that in $A_{r,s,\gamma}^{\rm st}$ the ideal generated by $\omega$ and $X$ admits divided powers for $\omega=p-[\underline{p}]$. As $A_{r,s,\gamma}^{st}$ is $p$-torsion free, it suffices to prove that $\omega^{p^n} \in p^n! \omega A_{r,s}$ and $X^{p^n}\in p^n! \gamma A_{r,s}^{\rm st}$. The $p$-adic valuation of $p^n!$ is $\frac{p^n-1}{p-1}$. As $p^n - \frac{p^n-1}{p-1} \geq p^{n-1}$ for $p\geq 3$  the second claim is clear. Next we  show that $\omega^{p^n} \in p^{\frac{p^n-1}{p-1}}  \omega A_{r,s}$. First $\omega^{p-1}\in (p,[\underline{p}])^{p-1} A_{\rm inf} \subset (p,  [\underline{p}]^{p-1}) A_{\rm inf} \subset  p  A_{r,s}$ as $\frac{[\underline{p}]^{p-1}}{p} \in A_{r,s} $ by assumption. Hence $\omega^p\in p \omega  A_{r,s}$.
Assuming that  $\omega^{p^n} \in p^{\frac{p^n-1}{p-1}} \omega  A_{r,s}$ one deduces that $\omega^{p^{n+1}} \in  p^{\frac{p^{n+1}-p}{p-1}}  \omega^p  A_{r,s} \subset  p^{\frac{p^{n+1}-p}{p-1}}  p \omega A_{r,s} $. As $p^{\frac{p^{n+1}-p}{p-1} +1}=p^{\frac{p^{n+1}-1}{p-1}}$ we conclude that $\omega^{p^{n+1}} \in p^{\frac{p^{n+1}-1}{p-1}} \omega  A_{r,s}$. This then holds for every $n$ by induction and implies the claim. 

For the second statement notice that the ideal ${\rm Ker}(\vartheta_{\rm st})$ is generated by the DP powers of $\omega$ and $X$. We have shown that the divided powers of $X$ lie in $ p A_{r,s,\gamma}^{\rm st}$ and that the divided powers of $\omega$ lie in $ \omega A_{r,s}\subset (p, [\underline{p}]) A_{r,s}$. As $(p, [\underline{p}]) A_{r,s} \subset ( [\underline{p}]^r,[\underline{p}])  A_{r,s}$ the claim follows.

The last statement follows remarking that  $[\underline{p}]^s= \frac{[\underline{p}]^s}{p} \cdot p$ in  $A_{r,s,\gamma}^{\rm st}$ and the latter is $p$-adically complete and separated.

(2), (3) and (4) are direct computations left to the reader using the explicit description of the rings involved.

\end{proof}

\begin{corollary}\label{corollary:invertiblematrix} Take $s\geq p-1 $.  Let $M$ be an $n\times n$ matrix with coefficients in  $\widehat{A}_{\rm st}$.   
\smallskip

(1) Let $\overline{M}$ be the reduction of $M$ modulo ${\rm Ker}(\vartheta_{\rm st})$. Assume that there exist  $r$ and $q\in \QQ$ with  $0\leq q<{\rm min}(r,1)$ and $r\leq s$ and an $n\times n$  matrix $\overline{N}$ with coefficients in $\cO_{\CC_p}$ such that $\overline{N} \overline{M}=p^q \mathbf{1}_n$. Then there exists an $n \times n$ matrix $N$ with coefficients in  $\widehat{A}_{r,s,\gamma}^{\rm st}$ such that  $M N=N M= [\underline{p}]^q \mathbf{1}_n$.

\smallskip
(2) Let $\overline{M}$ be the reduction of $M$ modulo $p$.   Assume that there exist $q$ and $r \in \mathbb{N}[1/p]$ with $0\leq q<r\leq s$  and an $n\times n$  matrix $\overline{N}$ with coefficients in $\widehat{A}_{\rm st}/p\widehat{A}_{\rm st}$ such that $\overline{N} \overline{M}=[\underline{p}]^q  \mathbf{1}_n$. Then there exists an $n \times n$ matrix $N$ with coefficients in  $\widehat{A}_{r,s,\gamma}^{\rm st}$ such that  $M N=N M= [\underline{p}]^q \mathbf{1}_n$.

\end{corollary}
\begin{proof} (1) Lift $\overline{N}$ to a matrix $N'$ with coefficients in $\widehat{A}_{\rm st}$. Then $N' M= [\underline{p}]^q \mathbf{1}_n + A$ with $A$ a matrix with coefficients in ${\rm Ker}(\vartheta_{\rm st})$. Then $A$ is a  matrix with coefficients in $ [\underline{p}]^{{\rm min}(r,1)}A_{r,s,\gamma}^{\rm st}$ by Lemma \ref{lemma:Astrs}. Hence, $ [\underline{p}]^q \mathbf{1}_n + A= [\underline{p}]^q B$ with $B$ a matrix with coefficients in $A_{r,s,\gamma}^{\rm st}$ congruent to the identity modulo  $[\underline{p}]^{{\rm min}(r,1)-q}$. Thus $B$ is invertible as its determinant is a unit by Lemma \ref{lemma:Astrs}. Take $N:= B^{-1} N'$.  

\smallskip

(2) As before lift $\overline{N}$ to a matrix $N'$ with coefficients in $\widehat{A}_{\rm st}$. Then $N' M= [\underline{p}]^q \mathbf{1}_n + A$ with $A$ with coefficients in $p \widehat{A}_{\rm st}$ and, hence, with coefficients in $ [\underline{p}]^{r}A_{r,s,\gamma}^{\rm st}$. Hence, $ [\underline{p}]^q \mathbf{1}_n + A= [\underline{p}]^q B$ with $B$ a matrix with coefficients in $A_{r,s,\gamma}^{\rm st}$ congruent to the identity modulo  $[\underline{p}]^{r-q}$. Thus $B$ is invertible and $N:= B^{-1} N'$ is the sought for matrix.

\end{proof}

\subsection{An approximation result}

The map $\vartheta\colon A_{\rm inf} \to \cO_{\CC_p}$ extends to a surjective map $\vartheta\colon A_{\rm inf}\bigl[[\underline{p}]^{-1}\bigr] \to \CC_p$ sending $[\underline{p}]\mapsto p$. Following the conventions in the paper we fix   $G_{\ZZ_p}\subset {\rm GSp}_{2g,\ZZ_p}$ a reductive subgroup and a parabolic subgroup $\widetilde{P}\subset G_{\cO_{\CC_p}}$ containing a Borel subgroup $\widetilde{B}\subset \widetilde{P}$ defined over $\ZZ_p$. We let $\widetilde{L}\subset \widetilde{P}$ be the induced standard Levi subgrpoup. We write $B$, $L$, $P$ for the induced subgroups of $G_{\CC_p}$.

\begin{lemma}\label{lemma:approximate} Let $h\in G_{\ZZ_p}( A_{\rm inf} )$ and let $\overline{A} \in {^h\widetilde{P}}(\CC_p)$ be an element such that, viewed as a matrix in ${\rm GSp}_{2g}\bigl(\CC_p\bigr)$,  it  has coefficients in ${\rm M}_{2g \times 2g}\bigl(\cO_{\CC_p}\bigr)$. Then, there exist matrices  $a\in {^h\widetilde{L}}(A_{\rm inf})$ and $A'\in {^h\widetilde{B}}\bigl(A_{\rm inf}\bigl[[\underline{p}]^{-1}\bigr]\bigr)$ such that $A'$, viewed as matrix in ${\rm GSp}_{2g}\bigl(A_{\rm inf}\bigl[[\underline{p}]^{-1}\bigr]\bigr)$, has coefficients in ${\rm M}_{2g \times 2g}\bigl(A_{\rm inf}\bigr)$  and $(A' \cdot a)\equiv \overline{A}$ modulo ${\rm Ker}(\vartheta)$. 

\end{lemma}
\begin{proof} We prove the statement for $h=\mathbf{1}$ and $\overline{A} \in \widetilde{P}(\CC_p)$. The other statement is proven from this by conjugating by $h$.

Decompose $\overline{A}=u \cdot \ell$ with $u\in {\rm R}_u P(\CC_p)$ (the unipotent radical of $P$) and $\ell \in L(\CC_p)$. Let $\widetilde{B}_L:=\widetilde{B}\cap \widetilde{L}$ be the Borel subgroup of $\widetilde{L}$ defined by $\widetilde{B}$. By the properness of the Grassmanian variety $\widetilde{B}_L\backslash \widetilde{L}$ we have that $$\widetilde{B}_L(\CC_p) \backslash \widetilde{L}(\CC_p)=\bigl(\widetilde{B}_L\backslash \widetilde{L}\bigr)(\CC_p)=\bigl(\widetilde{B}_L\backslash \widetilde{L}\bigr)\bigl(\cO_{\CC_p}\bigr) =\widetilde{B}_L(\cO_{\CC_p}) \backslash \widetilde{L}(\cO_{\CC_p}).$$Hence we can write $\ell \in L(\CC_p)=\widetilde{L}(\CC_p)$ as $u' \cdot \overline{a}$ with $u'\in \widetilde{B}_L(\CC_p)$ and  $\overline{a}\in  \widetilde{L}(\cO_{\CC_p})$. As the map $ \widetilde{L}\bigl(A_{\rm inf}\bigr)\to \widetilde{L}(\cO_{\CC_p})$ is surjective by the smoothness of $ \widetilde{L}$ we can lift $\overline{a}$ to an element $a\in  \widetilde{L}\bigl(A_{\rm inf}\bigr)$. 

Define $\overline{A}':=u u' \in B(\CC_p)$.  We are left to show that we can lift $\overline{A}'$ to an element in $A' \in \widetilde{B}\bigl(A_{\rm inf}\bigl[[\underline{p}]^{-1}\bigr]\bigr)$ which has coefficeints in $A_{\rm inf}$ when viewed in  ${\rm M}_{2g \times 2g}\bigl(A_{\rm inf}\bigl[[\underline{p}]^{-1}\bigr]\bigr)$. As $\widetilde{B}$ over $\cO_{\CC_p}$  is the semidirect product of a split torus $\widetilde{T}$ and a unipotent group $\widetilde{U}$, and the latter is a split extension of $\mathbb{G}_a$'s, we are left to prove the statement for  $\mathbb{G}:=\mathbb{G}_m$ and $\mathbb{G}:=\mathbb{G}_a$. Given a non-zero element $s \in\mathbb{G}(\CC_p)\subset \CC_p$, let $\alpha$ be its $p$-adic valuation. Then $s=p^{\alpha} s_o$ with $s_0\in \mathbb{G}\bigl(\cO_{\CC_p}\bigr)$. As $\mathbb{G}\bigl(A_{\rm inf}\bigr)\to \mathbb{G}\bigl(\cO_{\CC_p}\bigr)$ is surjective, we can lift $s_0$ to an element $\widetilde{s}_0\in \mathbb{G}\bigl(A_{\rm inf}\bigr)$ and then $\widetilde{s}:=\bigl[\underline{p}^{\alpha}\bigr]\widetilde{s}_0\in \mathbb{G}\bigl(A_{\rm inf}\bigl[[\underline{p}]^{-1}\bigr]\bigr)$ lifts $s$ as wanted. We need to show that the lift can be taken in ${\rm M}_{2g \times 2g}\bigl(A_{\rm inf}\bigr)$. Notice that this can be proven after conjugating by an element of ${\rm GL}_{2g}(\cO_{\CC_p} )$. 

We claim that, after possibly conjugating by an element of ${\rm GL}_{2g}(\cO_{\CC_p} )$ , we may assume that $\widetilde{B} \subset \widetilde{B}^{\rm std}$, with  $\widetilde{B}^{\rm std}\subset {\rm GL}_{2g,\ZZ_p}$ the standard Borel subgroup of upper tringular matrices, and that $\widetilde{T}$ is contained in the standard torus of  $\widetilde{B}^{\rm std}$. Indeed, extend $\widetilde{B}$ to a Borel subgroup $\widetilde{B}'$ of ${\rm GL}_{2g,\cO_{\CC_p}}$: first extend it over $\overline{\QQ}_p $ by extending the connected solvable subgroup  $\widetilde{B}_{\QQ_p}$ to a maximal connected solvable subgroup  $\widetilde{B}_{\overline{\QQ}_p}'$ of ${\rm GL}_{2g,\overline{\QQ}_p}$.  Since the scheme of Borel subgroups $\widetilde{B}_{\ZZ_p}^{\rm std}\backslash {\rm GL}_{2g,\ZZ_p} $ is proper, the $\overline{\QQ}_p$-valued point defined by $\widetilde{B}_{\overline{\QQ}_p}'$ extends to a $\overline{\ZZ}_p$-valued point so that, after possibly conjugation by an element of $ {\rm GL}_{2g}(\overline{\ZZ}_p)$, we may assume that $\widetilde{B}_{\QQ_p} \subset \widetilde{B}_{\QQ_p}^{\rm std}$  and, hence,  $\widetilde{B} \subset \widetilde{B}^{\rm std}$. As maximal tori in $\widetilde{B}^{\rm std}$ are conjugate over $\cO_{\CC_p}$ we may also assume that $\widetilde{T}$ is contained in the standard torus of  $\widetilde{B}^{\rm std}$ of diagonal matrices. 


By our assumption we now have that  $\overline{A}'$ lies in the upper triangular matrices in ${\rm M}_{2g \times 2g}\bigl(\cO_{\CC_p}\bigr)$. Consider the components $\overline{A}'_1$  on $\widetilde{T}(\CC_p)$ and $\overline{A}'_2$ in  $\widetilde{U}(\CC_p)$. By construction they  lie in the diagonal matrices and the upper triangular matrices in ${\rm M}_{2g \times 2g,\cO_{\CC_p}}$ respectively.  In particular, the  various projections of $\overline{A}'_2$ via the  decomposition (as schemes)  $\widetilde{U}\cong \mathbb{G}_a^r$ all lie in $\mathbb{G}_a(\cO_{\CC_p})$ and hence $\overline{A}'_2$ can be lifted to a matrix $A_2'$ in  $\widetilde{U}\bigl(A_{\rm inf}\bigr) \subset {\rm M}_{2g \times 2g}\bigl(A_{\rm inf}\bigr)$.

Fix an isomorphism $\widetilde{T}\cong \mathbb{G}_m^d$ and lift  $\overline{A}'_1=(s_1',\ldots, s_d')$ to an element $A_1'$ of  $\widetilde{T}\bigl[A_{\rm inf}[\underline{p}]^{-1}\bigr]\bigr)$ as explained above. Its image in ${\rm GL}_{2g}$ is a diagonal matrix of the form  
$\bigl(\bigl[\underline{p}^{\alpha_1}\bigr]\widetilde{s}_1,\ldots, \bigl[\underline{p}^{\alpha_{2g}}\bigr]\widetilde{s}_{2g}\bigr)$  with  $\widetilde{s}_1,\ldots, \widetilde{s}_{2g} \in \mathbb{G}_m (A_{\rm inf})$ and $\alpha_1,\ldots,\alpha_{2g}\in \QQ$. 
As the image of $\overline{A}'_1$ in the diagonal matrices of  ${\rm GL}_{2g}$ lies in  ${\rm M}_{2g \times 2g}\bigl(\cO_{\CC_p}\bigr)$ then each $\alpha_i$ must be non-negative.  Then $A':=A_1' \cdot A_2'$ is a lift of $\overline{A}'$ in $\widetilde{B}\bigl(A_{\rm inf}\bigl[[\underline{p}]^{-1}\bigr]\bigr)$ with the required property.

\end{proof}

\section{Appendix II: some facts about Weyl groups}\label{sec:Weyl}

We recall some facts from \cite[\S 1.4]{GK} and  \cite[\S 3.5]{Fzips}. Let $H$ be a connected reductive group over an algebariclly closed field $k$.  Fix a maximal torus $T$ and a Borel $B$. Let $P$ be a parabolic subgroup of $H$ of type $I$ and let $J:={^{w_0} I}$ be the type opposite to $I$.  Every element $w$ in $W$ can be written uniquely as $w=  {w^J} w_J$ for $w_J\in W_J$ and $w^J\in W^J$ and then $\ell(w)= \ell(w_J)+\ell( w^J)$. Similarly we have a unique decomposition $w=  w_I  {^I w}$ for $w_I\in W_I$ and ${^Iw}\in {^IW}$ and then $\ell(w)= \ell(w_I)+\ell({^I w})$. The groups $W$, $W_J$ and $W_I$ contain  unique elements of maximal length, denoted by $w_0$, $w_{J,0}$, $w_{I,0}$ respectively. By maximality they are elements of order $2$ and $w_0$ has the property that it sends positive roots to negative roots and that for every $w\in W$ $$\ell(w w_0)=\ell(w_0 w)=\ell(w_0)-\ell(w).$$ We also deduce from the uniqueness of the decomposition above  that in $W^J$ and in ${^IW}$ we have unique elements of maximal length $w_0^J\in W^J$ and ${^Iw}_0\in {^IW}$, characterized by the property that  $w_{I,0} {^Iw_0}=w_0=w_0^J w_{J,0} $. Conjugation by $w_0$ defines an involution $\iota_0\colon W\to W$ preserving the length of elements. In particular, it preserves the set of simple reflections $\Delta$.   Then,

\begin{lemma}\label{lemma:z} The map $\iota_0\colon W\to W$, $w\mapsto w_0 w w_0$, has the following properties 

\begin{itemize}

\item[a.] $\iota_0$ identifies $W_J=\iota_0(W_I)$ and $J=\iota_0(I)$. In particular, it induces mutually inverse and length preserving isomorphisms  $W_J\to W_I$ and $W_I\to W_J$;

\item[b.] $\iota_0$ identifies $W^J=\iota_0(W^I)$ (resp. ${^J W}=\iota_0({^IW})$) and induces mutually inverse and length preserving bijections  $W^J\to W^I$ and $W^I\to W^J$ (resp. ${^JW}\cong {^IW}$);

\item[c.] $\iota_0$ identifies  $w_{I,0}=\iota_0(w_{J,0})=w_0 w_{J,0} w_0$ and, in particular, we have   $w_0^J={^Iw}_0$ and  $w_0^I={^Jw}_0$. 

\end{itemize}

\end{lemma} 
\begin{proof} All assertions are trivial except for the identity  $w_0^J={^Iw}_0$ that follows from 
$$ {^Iw_0}=w_{I,0} w_0= w_0 w_{J,0} w_0 w_0=w_0 w_{J,0}=w_0^J.$$

\end{proof}

\begin{lemma}\label{lemma:dualityonWJ} (1)  The map $W\to W$, sending $w\to w^{-1}$, is an order and legth preserving bijection.   If $w=w^I w_I$ with $w^I\in W^I$ and $w_I\in W_I$ then $w^{-1}=w_I^{-1} (w^I)^{-1}$ and $w_I^{-1}\in W_I$ and $ (w^I)^{-1}= {^I w^{-1}}\in {^I W}$.

\

(2) The map $W\to W$, sending $w\mapsto w_0 w w_{I,0}$, induces an order reversing bijection $W^I\to W^I$ such that $\ell(w_0 w w_{I,0})=\ell(w_0^I)-\ell(w)$. 

Similarly the map $W\to W$, sending $w\mapsto w_{I,0} w w_0$, induces an order reversing bijection ${^IW}\to {^IW}$ such that $\ell(w_{I,0} w w_0)=\ell(w_0^I)-\ell(w)$.

\

(3) The map $\sigma_0\colon W\to W$, sending $w \mapsto {w_0^I} w$, satisfies the following properties:

\begin{itemize}

\item[i.] it is a bijection and it sends bijectively  ${^IW}\to {^JW}$;
\item[ii.] it satisfies  $\ell({w_0^I} w)=\ell({w_0^J})-\ell(w)$ for every $w\in {^IW}$;
\item[iii.] it reverses the orders $\preccurlyeq$ on ${^IW}$ and ${^JW}$,  defined as follows. Let $\leq$ be the Bruhat order on $W$ and let $z=\varphi(w_0^I)$. For $w$ and $w'\in {^IW}$ (resp.~in ${^J}W$)  we say that $w' \preccurlyeq w$ iff there exists $u\in W_I$ (resp.~$u\in W_J$) such that $u w' z \varphi(u)^{-1} z^{-1}\leq w$   (resp.~$u w' \varphi(u)^{-1}\leq w$) .
\end{itemize}

\end{lemma}
\begin{proof} (1) is clear. 

(2) We prove the first statement. The second is obtained from the first using the order and length preserving bijection $W^I \to {^IW}$ given by $w\mapsto w^{-1}$ obtained from (1). First we compute lengths. For every $w\in W^I$ and every $y\in W_I$ we have $\ell(w y)=\ell(w) + \ell(y)$. Then $\ell(w_0 w y)=\ell(w_0)-  \ell(y) -\ell(w) $. So for $y=w_{I,0}$ we get that $\ell(w_0 w w_{I,0})=\ell(w_0)-  \ell( w_{I,0}) -\ell(w)=\ell(w_0^I)-\ell(w)$ as wanted. On the other hand the function $y\mapsto \ell(w_0 w y)$ is minimized for $\ell(y)$ maximal, i.e., for $y=w_{I,0}$. Hence, $w_0 w w_{I,0} $ is the element in  the coset $w_0 w W_I$ of minimal length, i.e., $w_0 w w_{I,0} \in W^I$.

(3) As $w_0 w_{I,0}=w_0$, the given map is the composite of the map $W\to W$, $w\mapsto w_{I,0} w w_0$, and of the map $\iota_0\colon W\to W$, $w\mapsto w_0 w w_0$. These are both involutions and hence  $\sigma_0$ is a bijection on $W$. It sends bijectively ${^IW}$ to ${^JW}$, reversing the Bruhat orders $\leq $, thanks to Claim (2) and Lemma \ref{lemma:z}.

Write $y:=w_0^I$ for simplicity. Take $w$ and $w'\in {^IW}$ such that  $w' \preccurlyeq w$ and let $u\in W_I$ be such that $u w' z \varphi(u)^{-1} z^{-1}\leq w$. This implies that $u w' z \varphi(u)^{-1} z^{-1}\in {^IW}$. We deduce that  $ 
y u w' z \varphi(u)^{-1} z^{-1}=(y u y^{-1}) (y w') \varphi(y u y^{-1})^{-1}  \geq y w$. Write $t:=y u^{-1} y^{-1}$. Then  $t (y w) \varphi(t)^{-1} \leq y w'$. Note that $y u y^{-1}=w_0^I u {^Iw}_0=w_0 \bigl(w_{I,0} u w_{I,0}\bigr) w_0 $ as $y^{-1}={^I w_0 }$ by Claim (1). Since $w_{I,0} u w_{I,0}\in W_I$, we conclude that $t \in W_J$ thanks to  Lemma \ref{lemma:z} so that  $y w\preccurlyeq  y w'$.

\end{proof}

\end{document}